\documentclass[10 pt,english]{article}
\usepackage[activeacute]{babel}
\usepackage{amsfonts} 
\usepackage{mathrsfs} 
\usepackage{pstricks-add}
\usepackage{amssymb, amsmath, amsfonts}
\usepackage{makeidx}
\usepackage{epsfig}
\usepackage{color}
\usepackage[T1]{fontenc}
\usepackage{pstricks, pst-node}
\usepackage{graphics,graphicx,graphpap}
\usepackage[ansinew]{inputenc}
\usepackage{fancyhdr}
\usepackage{color}
\usepackage{amsthm}
\usepackage{multirow}

\if@twoside \oddsidemargin 5pt \evensidemargin 5pt \marginparwidth 20pt
 \else \oddsidemargin 5pt \evensidemargin 5pt \marginparwidth 20pt \fi

\newcommand{\ZZ}{\mathbb{Z}}

\newcommand{\Aut}{\mathrm{Aut}}
\newcommand{\G}{\Gamma}

\newcommand{\N}{\mathcal{N}}
\newcommand{\oo}{\mathcal{O}}
\newcommand{\ot}{\leftarrow}

\newtheorem{theorem}{Theorem}[section]
\newtheorem{proposition}[theorem]{Proposition}

\newtheorem{lemma}[theorem]{Lemma}
\newtheorem{problem}[theorem]{Problem}
\newtheorem{question}[theorem]{Question}

\theoremstyle{definition}

\newtheorem{construction}[theorem]{Construction}

\begin{document}

\begin{center}
\Large{\textbf{On certain edge-transitive bicirculants}} \\ [+4ex]
Robert Jajcay{\small$^{a, b,}$\footnotemark}, \v Stefko Miklavi\v c{\small$^{b, d,}$\footnotemark}, Primo\v z \v Sparl{\small$^{b, c, d,}$\footnotemark$^{,*}$}, Gorazd Vasiljevi\'c
\\ [+2ex]
{\it \small 
$^a$Comenius University, Bratislava, Slovakia\\
$^b$University of Primorska, Institute Andrej Maru\v si\v c, Koper, Slovenia\\
$^c$University of Ljubljana, Faculty of Education, Ljubljana, Slovenia\\
$^d$Institute of Mathematics, Physics and Mechanics, Ljubljana, Slovenia}
\end{center}

\addtocounter{footnote}{-2}
\footnotetext{Supported in part by VEGA 1/0474/15, VEGA 1/0596/17, NSFC 11371307, APVV-15-0220, and by the Slovenian Research Agency research project J1-6720.}
\addtocounter{footnote}{1} \footnotetext{Supported in part by the Slovenian Research Agency, program P1-0285 and research projects N1-0032, N1-0038, J1-6720, J1-7051.}
\addtocounter{footnote}{1} \footnotetext{Supported in part by the Slovenian Research Agency, program P1-0285 and research projects N1-0038, J1-6720, J1-7051.

Email addresses: 
robert.jajcay@fmph.uniba.sk (Robert Jajcay),
stefko.miklavic@upr.si (\v Stefko Miklavi\v c),
primoz.sparl@pef.uni-lj.si (Primo\v z \v Sparl),
gorazd.vasiljevic@gmail.com (Gorazd Vasiljevi\'c).

~*Corresponding author  }


\hrule

\begin{abstract}
A graph $\G$ of even order is a {\em bicirculant} if it admits an automorphism with two orbits of equal length. Symmetry properties of bicirculants, for which at least one of the induced subgraphs on the two orbits of the corresponding semiregular automorphism is a cycle, have been studied, at least for the few smallest possible valences. For valences $3$, $4$ and $5$, where the corresponding bicirculants are called generalized Petersen graphs, Rose window graphs and Taba\v cjn graphs, respectively, all edge-transitive members have been classified. While there are only 7 edge-transitive generalized Petersen graphs and only 3 edge-transitive Taba\v cjn graphs, infinite families of edge-transitive Rose window graphs exist. The main theme of this paper is the question of the existence of such bicirculants for higher valences. It is proved that infinite families of edge-transitive examples of valence $6$ exist and among them infinitely many arc-transitive as well as infinitely many half-arc-transitive members are identified. Moreover, the classification of the ones of valence $6$ and girth $3$ is given. As a corollary, an infinite family of half-arc-transitive graphs of valence $6$ with universal reachability relation, which were thus far not known to exist, is obtained.
\end{abstract}

\hrule

\begin{quotation}
\end{quotation}

\section{Introduction}
\label{sec:Intro}

Even though almost all graphs have no nontrivial automorphisms (see for instance \cite[Corollary~2.3.3]{GodRoy_book}) investigation of highly symmetric graphs has been a very active topic of research in algebraic graph theory for decades. The majority of several hundreds of papers on this topic have focused on graphs with a particular degree of symmetry, such as vertex-transitivity, edge-transitivity or arc-transitivity (see Section~\ref{sec:prelim} for the definitions). Since the family of all arc-transitive graphs (let alone the families of edge-transitive or vertex-transitive graphs) is way too rich to be investigated as a whole, one has to restrict to some specific subfamily to be able to obtain classification type results. For instance, there are several papers giving a classification of cubic or tetravalent arc-transitive graphs of specific types of orders (see for instance~\cite{FenKwa07, ZhoFen10} and the references therein), or with other restrictions such as their girth (see for instance~\cite{KutMar09}). 

When dealing with graphs with a high degree of symmetry the following viewpoint is of interest. In 1981 Maru\v si\v c conjectured~\cite{Mar81} that every vertex-transitive graph admits a nontrivial {\em semiregular} automorphism (that is, an automorphism having all orbits of the same length). The conjecture, now known as the {\em Polycirculant conjecture}, is still open, but several results confirming the conjecture for some restricted subfamilies have been obtained (see for instance~\cite{DobMalMarNow07, Ver15}). Now, the nicest possibility regarding the existence of semiregular automorphisms is that the semiregular automorphism has just one orbit. In this case the graph is a Cayley graph of a cyclic group, a so called {\em circulant}. These graphs are quite well understood. For instance, arc-transitive circulants have been characterized independently by Kov\'acs and Li~\cite{Kov04, Li05}. Since each edge-transitive Cayley graph of an Abelian group is automatically arc-transitive, this in fact gives a characterization of all edge-transitive circulants.

The next best possibility is that the graph admits a semiregular automorphism with two orbits. Such graphs are called {\em bicirculants}. Even though we are currently nowhere near such general results on arc-transitive, let alone edge-transitive, bicirculants as the ones from~\cite{Kov04, Li05}, some progress has been made. For instance, the automorphism groups of bicirculants, for which the two orbits of the semiregular automorphism are of prime length, are quite well understood~\cite{MalMarSpaFre07}. Classification results for arc-transitive bicirculants of small valences have also been obtained. For instance, combining together the results of~\cite{FruGraWat71, MarPis00, Pis07} one obtains the classification of all cubic arc-transitive bicirculants. Similarly, the classification of tetravalent arc-transitive bicirculants is obtained by combining the results of~\cite{KovKutMar10, KovKuzMal10, KovKuzMalWil12}. Recently, the classification of pentavalent arc-transitive bicirculants was also obtained~\cite{AntHujKut15, ArrHubKutOreSpa15}. 

One of the most important steps in these classifications is the classification of the arc-transitive bicirculants (which for these valences actually coincide with the edge-transitive ones), for which at least one of the induced subgraphs on the two orbits of the corresponding semiregular automorphism is a cycle. The corresponding graphs for valences 3, 4 and 5 are called {\em generalized Petersen} graphs, {\em Rose window} graphs and {\em Taba\v cjn} graphs, respectively. The edge-transitive members of these three families of graphs were classified in~\cite{FruGraWat71}, \cite{KovKutMar10} and \cite{ArrHubKutOreSpa15}, respectively. It is interesting to note that while there are only 7 edge-transitive generalized Petersen graphs and only 3 edge-transitive Taba\v cjn graphs infinite families of edge-transitive Rose window graphs exist (which was first pointed out by Wilson~\cite{Wil08} when he introduced the Rose window graphs). It is thus very natural to ask whether edge-transitive analogues of these graphs of higher valences also exist. The $6$-valent analogues, which are obtained from the generalized Petersen graphs by adding three additional perfect matchings between the two orbits of the corresponding semiregular automorphism (see Section~\ref{sec:Nest} for the formal definition), were first studied by Vasiljevi\'c~\cite{Vas17} who named them {\em Nest} graphs. In this paper we show that infinite families of arc-transitive, as well as of half-arc-transitive (see Section~\ref{sec:prelim} for the definition) Nest graphs exist. It should be pointed out that the members of the infinite family of graphs that has quite recently been obtained by Zhou and Zhang~\cite{ZhoZha17} when they classified half-arc-regular bicirculants of valence 6 also turn out to be Nest graphs. The existence of infinitely many edge-transitive Nest graphs thus motivates the following question.

\begin{question}
\label{que:main}
For which integers $d > 6$ does there exist an edge-transitive bicirculant of valence $d$, such that at least one of the subgraphs induced on the two orbits of the corresponding semiregular automorphism is a cycle? For which of these valences do infinitely many such examples exist?
\end{question}

For small valences one can search for examples using a computer. An exhaustive computer search shows that there exists no edge-transitive bicirculant of valence $d$ where $7 \leq d \leq 10$ and order at most $100$, such that at least one of the subgraphs induced on the two orbits of the corresponding semiregular automorphism is a cycle. Due to the fact that the seven cubic examples have orders 8, 10, 16, 20, 20, 24 and 48, the three pentavalent examples have orders 6, 12 and 12, while in the cases of valence 4 and 6 we have examples of almost every even order starting from $6$ and $8$, respectively, it very well might be the case that the answer to Question~\ref{que:main} is that there are in fact no such examples. This implies that the following natural problem might be quite important.

\begin{problem}
\label{pro:main}
Classify the edge-transitive Nest graphs.
\end{problem}

We give a partial solution to this problem by classifying the examples of girth $3$ (see Theorem~\ref{the:girth3}). Since the girth of any Nest graph is at most $6$ (see Section~\ref{sec:Nest}) this leaves the girths $4$, $5$ and $6$ to be dealt with. We finish this section by highlighting another result of this paper. When dealing with half-arc-transitive graphs the reachability relation and the corresponding alternets (see Section~\ref{sec:prelim} for the definitions) play an important role. Namely, the alternets give an insight into the structure of the graph in question, and, in the case that we have more than one alternet, give rise to imprimitivity block systems for the corresponding automorphism group. The situation when one has just one alternet, that is when the reachability relation is universal, thus deserves special attention (we remark that half-arc-transitive graphs with a few alternets were studied in~\cite{HujKutMar14}). In 2010 an infinite family of half-arc-transitive graphs of valence 12 with universal reachability relation was constructed~\cite{KutMarSpa10}. Until now, this was the smallest valence for which a half-arc-transitive graph with universal reachability relation was known to exist. Since Maru\v si\v c proved~\cite{Mar98} that the reachability relation cannot be universal in a half-arc-transitive graph of valence $4$, the smallest possible valence for which a half-arc-transitive graph with universal reachability relation could exist is $6$. In Theorem~\ref{the:universal} we prove that $6$ is indeed attained by exhibiting an infinite family of half-arc-transitive graphs of valence $6$ with universal reachability relation.

\section{Notation and definitions}
\label{sec:prelim}

Throughout the paper the graphs are assumed to be finite and undirected, even though we will occasionally be working with an orientation of the edges of the graph, implicitly given by the action of its automorphism group. For a graph $\G$ and its vertex $x$ the neighborhood of $x$ in $\G$ will be denoted by $\G(x)$, while the fact that the vertices $x$ and $y$ are adjacent in $\G$ will be denoted by $x \sim y$. Throughout the paper we will often be working with $2$-paths of the graph in question. We point out that, unless otherwise specified, we consider a $2$-path $(x,y,z)$ simply as a subgraph of the graph in question, and so we consider the $2$-paths $(x,y,z)$ and $(z,y,x)$ as being equal.

For an integer $n$ the residue class ring modulo $n$ will be denoted by $\ZZ_n$. Throughout the paper we will constantly be working with integers and elements from $\ZZ_n$ and will sometimes regard them simply as integers while at other times as elements of $\ZZ_n$. For instance, if for $1 \leq b,k \leq n-1$ we write $b+2k = 1$ we mean that when $b$ and $k$ are viewed as elements of $\ZZ_n$ equality $b+2k = 1$ holds (in $\ZZ_n$). On the other hand, if we write $b+k < n$ we mean that the sum of integers $b$ and $k$ is strictly smaller than $n$ (of course without making the calculation modulo $n$). This should cause no confusion but we will nevertheless sometimes stress that we want to view a certain expression within $\ZZ_n$ to make things completely unambiguous.

A subgroup $G \leq \Aut(\G)$ of the automorphism group of the graph $\G$ is said to be {\em vertex-transitive}, {\em edge-transitive} and {\em arc-transitive}, respectively, if the induced action of $G$ on the vertex set, edge set and arc set, respectively, is transitive. If $G$ is vertex- and edge-transitive but not arc-transitive, it is {\em half-arc-transitive}. When $G =  \Aut(\G)$ in the above definitions we say that $\G$ is vertex-, edge-, arc- or half-arc-transitive, respectively. 

It is well known that in a half-arc-transitive graph $\G$ no automorphism can interchange a pair of adjacent vertices (see for instance~\cite[Proposition~2.1]{Mar98}). Moreover, the action of the automorphism group of $\G$ induces two paired natural orientations of the edges of $\G$, implying that $\G$ is of even valence. When dealing with half-arc-transitive graphs one usually fixes one of these two natural orientations of the edges of $\G$. The fact that the edge $xy$ of $\G$ is oriented from $x$ to $y$ in this orientation will be denoted by $x \to y$ or $y \ot x$ and the vertices $x$ and $y$ will be referred to as the {\em tail} and the {\em head} of the edge $xy$, respectively. Of course, for any vertex $x$, half of the edges, incident with $x$, have $x$ as their tail and half of them have $x$ as their head.

Suppose $\G$ is half-arc-transitive and fix one of the two natural $\Aut(\G)$-induced orientations of the edges. One can then define the {\em reachability relation} on the edge set of $\G$, first introduced in~\cite{CamPraWor93} in the context of infinite digraphs, as follows. An edge $f$ is {\em reachable} from the edge $e$ if there exists an alternating path (with respect to the fixed orientation) whose starting and terminal edges are $e$ and $f$. The reachability relation is clearly an equivalence relation and does not depend on which of the two paired $\Aut(\G)$-induced orientations of the edges one has chosen. Its equivalence classes are called {\em alternets}.

\section{The Nest graphs}
\label{sec:Nest}

In this section the Nest graphs are formally introduced and two infinite families of edge-transitive examples are identified. The Nest graphs are obtained from the generalized Petersen graphs by adding three additional perfect matchings between the two orbits of the natural $(2,n)$-semiregular automorphism, which are all consistent with its action.

\begin{construction}
Let $n \geq 4$ and $1 \leq a,b,c,k \leq n-1$ be integers such that $k \neq n/2$ and $a$, $b$ and $c$ are pairwise distinct. Then the {\em Nest graph} $\N(n;a,b,c;k)$ is the graph of order $2n$ with vertex set consisting of two sets of size $n$, namely $\{u_i \colon i \in \ZZ_n\}$ and $\{v_i \colon i \in \ZZ_n\}$, and edge set consisting of the following six sets of size $n$ (where computations are performed modulo $n$): 
\begin{itemize}\itemsep = 0pt
\item	the set $E_{rim}$ of {\em rim edges} $\{u_i u_{i+1} \colon i \in \ZZ_n\}$, 
\item	the set $E_{hub}$ of {\em hub edges} $\{v_i v_{i+k} \colon i \in \ZZ_n\}$, 
\item	the set $E_0$ of {\em $0$-spokes}  $\{u_i v_{i} \colon i \in \ZZ_n\}$, 
\item	the set $E_a$ of {\em $a$-spokes} $\{u_i v_{i+a} \colon i \in \ZZ_n\}$, 
\item	the set $E_b$ of {\em $b$-spokes}  $\{u_i v_{i+b} \colon i \in \ZZ_n\}$, 
\item	the set $E_c$ of {\em $c$-spokes}  $\{u_i v_{i+c} \colon i \in \ZZ_n\}$.
\end{itemize}

Observe that the assumptions on the parameters imply that the graph $\N(n;a,b,c;k)$ is indeed a regular graph of valence $6$. Moreover, it is clear that the graph admits the $(2,n)$-semiregular automorphism $\rho$, mapping according to the rule 
\begin{equation}
\label{eq:rho}
u_i\rho = u_{i+1}\quad \mathrm{and} \quad v_i\rho = v_{i+1}\quad \mathrm{for\ all\ }i \in \ZZ_n.
\end{equation}
\end{construction}

Unlike the generalized Petersen graphs and the Rose window graphs, all of which admit an additional involutory automorphism normalizing $\rho$, there are in general no additional automorphisms (other than the ones from the subgroup $\langle \rho \rangle$) which would automatically be ensured in the Nest graphs. For instance, it is easy to verify that the graph $\N(7;1,2,4;2)$ is such an example (it is in fact the smallest such graph). Nevertheless, the following holds.

\begin{lemma}
\label{le:c=a+b}
Let $\G = \N(n;a,b,c;k)$ be a Nest graph with $c = a + b$ in $\ZZ_n$. Then the permutation $\tau$ of $V(\G)$, given by the rule
\begin{equation}
\label{eq:tau}
	u_i\tau = u_{-i}\quad \mathrm{and}\quad v_i\tau = v_{-i+c}\quad \mathrm{for\ all\ }i \in \ZZ_n,
\end{equation}
is an automorphism of $\G$. Consequently, $\G$ is edge-transitive if and only if it is arc-transitive. 
\end{lemma}

\begin{proof}
Using the fact that $c = a+b$ it can easily be verified that $\tau$ preserves adjacency. Now, if $\G$ is edge-transitive, it is automatically vertex-transitive (since the automorphism $\rho$ from (\ref{eq:rho}) has just two orbits on the vertex set of $\G$ and some edges of $\G$ connect vertices from the same orbit of $\rho$ while other connect vertices from different orbits of $\rho$). Since $\tau\rho$ interchanges the pair $u_0$, $u_1$ of adjacent vertices it thus follows that $\G$ is arc-transitive.
\end{proof}

Of course, different sets of parameters $a,b,c,k$ for a fixed $n$ may result in isomorphic graphs. We record some rather obvious isomorphisms, which are best described intuitively by `reflecting' with respect to the edge $u_0v_0$ (that is, exchanging the roles of $u_i$ and $v_i$ by $u_{n-i}$ and $v_{n-i}$, respectively, for all $i$) or ``rotating'' the set of vertices of the form $v_i$ by $a$ steps (that is, renaming each $v_i$ by $v_{i-a}$).

\begin{lemma}
\label{le:isom}
Let $n \geq 4$ and $1 \leq a,b,c,k \leq n-1$ be integers such that $k \neq n/2$ and $a$, $b$ and $c$ are pairwise distinct. Then the graph $\N(n;a,b,c;k)$ is isomorphic to each of the graphs $\N(n;a',b',c';k)$, where $\{a',b',c'\} = \{a,b,c\}$, as well as to any of the graphs $\N(n;a,b,c;-k)$, $\N(n;-a,-b,-c;k)$ and $\N(n;-a,b-a,c-a;k)$.
\end{lemma}

The above lemma implies that we can assume $a < b < c$ and $k < n/2$. Moreover, ``rotating'' the vertices of the form $v_i$ by $a$, $b$ or $c$, if necessary, we can assume that $a$ is minimal among the elements of $\{a, b-a, c-b, n-c\}$. Unless otherwise specified we will always make this assumption.
\medskip

One of the goals of this paper is to investigate the edge-transitive Nest graphs. Using Lemma~\ref{le:isom} a computer search for all edge-transitive examples up to some reasonable order can be performed. In Table~\ref{tab:all} all pairwise nonisomorphic edge-transitive Nest graphs of order up to $220$ are given. For each of the graphs the defining parameters $n,a,b,c$ and $k$ are given, as well as its girth, an indication of whether or not the graph is bipartite, the size of the vertex-stabilizer and an indication of whether the graph is arc-transitive or half-arc-transitive. Observe that every edge-transitive Nest graph is automatically vertex-transitive, and so the edge-transitive graphs that are not arc-transitive are half-arc-transitive. Table~\ref{tab:all} reveals that, in contrast to the fact that there is no half-arc-transitive Rose window graph~\cite{KovKutMar10}, there do exist half-arc-transitive Nest graphs. 

\begin{table}
\centering
{\scriptsize
\begin{tabular}{|c|c|c|c|c||c|c|c|c|c|}
\hline
$(n;a,b,c;k)$ & girth & bip & stab & AT/HAT & $(n;a,b,c;k)$ & girth & bip & stab & AT/HAT \\
\hline
$(4;1,2,3;1)$ & $3$ & no & $48$ & AT  & 	$(54;9,25,34;1)$ & $4$ & no & $6$ & AT \\
$(5;1,2,3;2)$ & $3$ & no & $12$ & AT  & 	$(54;2,27,29;1)$ & $4$ & no & $12$ & AT \\
$(6;1,3,4;1)$ & $3$ & no & $12$ & AT  & 	$(58;2,29,31;1)$ & $4$ & no & $12$ & AT \\
$(8;1,3,4;3)$ & $3$ & no & $72$ & AT  & 	$(60;2,15,17;29)$ & $4$ & no & $6$ & AT \\
$(8;1,2,5;3)$ & $3$ & no & $12$ & AT  & 	$(62;2,31,33;1)$ & $4$ & no & $12$ & AT \\
$(10;2,5,7;1)$ & $4$ & no & $12$ & AT  & 	$(62;1,11,12;1)$ & $3$ & no & $6$ & AT \\
$(10;1,3,4;3)$ & $3$ & no & $12$ & AT  & 	$(66;2,33,35;1)$ & $4$ & no & $12$ & AT \\
$(10;2,4,6;3)$ & $4$ & yes & $12$ & AT  & 	$(68;2,17,19;33)$ & $4$ & no & $6$ & AT \\
$(12;1,3,10;5)$ & $3$ & no & $6$ & AT  & 	$(70;5,27,32;1)$ & $4$ & no & $6$ & AT \\
$(12;2,4,8;5)$ & $4$ & yes & $48$ & AT  & 	$(70;2,35,37;1)$ & $4$ & no & $12$ & AT \\
$(14;2,7,9;1)$ & $4$ & no & $12$ & AT  & 	$(74;1,21,22;1)$ & $3$ & no & $6$ & AT \\
$(14;1,5,6;1)$ & $3$ & no & $6$ & AT  & 	$(74;2,37,39;1)$ & $4$ & no & $12$ & AT \\
$(18;3,7,10;1)$ & $4$ & no & $6$ & AT  & 	$(76;2,19,21;37)$ & $4$ & no & $6$ & AT \\
$(18;2,9,11;1)$ & $4$ & no & $12$ & AT  & 	$(76;1,15,54;37)$ & $3$ & no & $3$ & HAT \\
$(20;2,5,7;9)$ & $4$ & no & $6$ & AT  & 	$(78;2,39,41;1)$ & $4$ & no & $12$ & AT \\
$(22;2,11,13;1)$ & $4$ & no & $12$ & AT  & 	$(78;1,33,34;1)$ & $3$ & no & $6$ & AT \\
$(26;1,7,8;1)$ & $3$ & no & $6$ & AT  & 	$(82;2,41,43;1)$ & $4$ & no & $12$ & AT \\
$(26;2,13,15;1)$ & $4$ & no & $12$ & AT  & 	$(84;1,10,51;41)$ & $3$ & no & $3$ & HAT \\
$(28;1,6,19;13)$ & $3$ & no & $3$ & HAT  & 	$(84;2,21,23;41)$ & $4$ & no & $6$ & AT \\
$(28;2,7,9;13)$ & $4$ & no & $6$ & AT  & 	$(86;1,13,14;1)$  & $3$ & no  & $6$ & AT\\
$(30;2,15,17;1)$ & $4$ & no & $12$ & AT  & 	$(86;2,43,45;1)$  & $4$ & no & $12$ & AT\\
$(34;2,17,19;1)$ & $4$ & no & $12$ & AT  & 	$(90;15,43,58;1)$ & $4$ & no & $6$ & AT \\
$(36;3,10,25;17)$ & $5$ & no & $3$ & HAT  & 	$(90;2,45,47;1)$ & $4$ & no & $12$ & AT \\
$(36;2,9,11;17)$ & $4$ & no & $6$ & AT  & 	$(92;2,23,25;45)$ & $4$ & no & $6$ & AT \\
$(38;2,19,21;1)$ & $4$ & no & $12$ & AT  & 	$(94;2,47,49;1)$ & $4$ & no & $12$ & AT \\
$(38;1,15,16;1)$ & $3$ & no & $6$ & AT  & 	$(98;1,37,38;1)$ & $3$ & no & $6$ & AT \\
$(42;2,21,23;1)$ & $4$ & no & $12$ & AT &	$(98;2,49,51;1)$ & $4$ & no & $12$ & AT \\
$(42;1,9,10;1)$ & $3$ & no & $6$ & AT &		$(100;2,25,27;49)$ & $4$ & no & $6$ & AT \\	
$(44;2,11,13;21)$ & $4$ & no & $6$ & AT &	$(102;2,51,53;1)$  & $4$ & no & $12$ & AT\\
$(46;2,23,25;1)$ & $4$ & no & $12$ & AT &	$(106;2,53,55;1)$  & $4$ & no & $12$ & AT\\
$(50;2,25,27;1)$ & $4$ & no & $12$ & AT  &	$(108;2,27,29;53)$ & $4$ & no & $6$  & AT\\
$(52;2,13,15;25)$ & $4$ & no & $6$ & AT	 &	$(108;9,34,79;53)$ & $6$ & no & $3$  & HAT\\
$(52;1,7,34;25)$ & $3$ & no & $3$ & HAT  & 	$(110;2,55,57;1)$ &  $4$ & no & $12$ & AT\\
\hline
\end{tabular}
}
\caption{All edge-transitive Nest graphs of order up to $220$.}
\label{tab:all}
\end{table}

Table~\ref{tab:all} also seems to suggest that the family of edge-transitive, as well as the family of arc-transitive Nest graphs, is infinite. We prove this by exhibiting an infinite family of examples in Lemma~\ref{le:fam1}. Two additional infinite families of arc-transitive Nest graphs are given in Lemma~\ref{le:fam2} and Lemma~\ref{le:lambda2_fam}, while an infinite family of half-arc-transitive Nest graphs is given in Proposition~\ref{pro:lambda1}. There are various other observations to be made. For instance, except for the graph of order $10$ (which happens to be the complement of the Petersen graph) the order of all edge-transitive Nest graphs seems to be divisible by $4$ (that is, $n$ is even) but, except for the two graphs of order $16$, none of them seems to have order divisible by $16$. Next, the vertex-stabilizers seem to be bounded from above by $12$ (except for the three graphs of orders $8$, $16$ and $24$, respectively). To mention just two more things, it appears that, except for two graphs of orders $20$ and $24$, none of the examples is bipartite and, except possibly for a very specific family of half-arc-transitive examples whose orders are an odd multiple of $72$, all examples are either of girth $3$ or $4$. 

It is not difficult to show that the girth of (edge-transitive) Nest graphs is at most $6$. Namely, every Nest graph $\N(n;a,b,c;k)$ contains several $6$-cycles. For instance, since the assumption $1 \leq a < b < c \leq n-1$ implies that $b$ is different from both $1$ and $-1$, we have the $6$-cycle $(v_0,u_0,u_1,v_1,u_{1-b},u_{-b})$. On the other hand, the example $\N(108;9,34,79;53)$ from Table~\ref{tab:all} shows that the girth of an (edge-transitive) Nest graph can actually be equal to $6$. Since the only two graphs from Table~\ref{tab:all}, whose girth exceeds $4$, are not arc-transitive, it is an interesting question to ask whether an arc-transitive Nest graph with girth greater than $4$ exists. \medskip

Based on the data from Table~\ref{tab:all} it is easy to identify infinite families of edge-transitive Nest graphs. We present two in the next two lemmas. 

\begin{lemma}
\label{le:fam1}
Let $m \geq 3$ be an odd integer. Then the Nest graph $\N(2m;2,m,m+2;1)$ is arc-transitive having vertex-stabilizers of order $12$.
\end{lemma}

\begin{proof}
Let $\G = \N(2m;2,m,m+2;1)$ and let $\varphi$ be the permutation of the vertex set of $\G$, given by the rule
$$
	u_i\varphi = \left\{\begin{array}{lcl}u_{-i} & ; & i \mathrm{\ even}\\
				v_{-i+1} & ; & i \mathrm{\ odd},\end{array}\right. \quad \quad
	v_i\varphi = \left\{\begin{array}{lcl}u_{-i+1} & ; & i \mathrm{\ even}\\
				v_{-i+2} & ; & i \mathrm{\ odd},\end{array}\right. 				
$$
where $i \in \ZZ_n$. 
It is easy to verify that $\varphi$ is an automorphism of $\G$. We give some details and leave the rest to the reader. For instance, the rim edge $u_iu_{i+1}$ is mapped to the $0$-spoke $u_{-i}v_{-i}$ or the $2$-spoke $v_{-i+1}u_{-i-1}$, depending on whether $i$ is even or odd, respectively. Similarly, since $m$ is odd, the $m$-spoke $u_iv_{m+i}$ is mapped to the $(m+2)$-spoke $u_{-i}v_{-i+m+2}$ or the $m$-spoke $u_{-i+1+m}v_{-i+1}$, depending on whether $i$ is even or odd, respectively. 

Note that Lemma~\ref{le:c=a+b} implies that the permutation $\tau$ defined in (\ref{eq:tau}) is an automorphism of $\G$. It is also easy to see that the permutation $\eta$, defined by the rule
$$
	u_i\eta = u_i\quad \mathrm{and} \quad v_i\eta = v_{i+m}\quad \mathrm{for\ all\ }i \in \ZZ_n,
$$
is an involutory automorphism of $\G$. It is now clear that the group $G = \langle \rho, \varphi, \tau, \eta\rangle$, where $\rho$ is as in (\ref{eq:rho}), acts arc-transitively on $\G$. We finally prove that in fact $G = \Aut(\G)$ and that the vertex-stabilizers in $G$ are of order $12$. In view of arc-transitivity it suffices to prove that the only nontrivial element of the pointwise stabilizer of the arc $(u_0,u_1)$ is the involution $\eta$. To see this let $\psi$ be any nontrivial automorphism of $\G$ fixing both $u_0$ and $u_1$. It is easy to see that of the five remaining neighbors of $u_0$ (other than $u_1$) $u_{-1}$ is the only one having exactly three common neighbors with $u_1$ (namely, $u_0$, $v_1$ and $v_{m+1}$). It thus follows that $\psi$ fixes $u_{-1}$ as well. Repeating the same argument for $u_{-1}$ and $u_0$ one can see that $\psi$ fixes $u_{-2}$, and inductively that it fixes each $u_i$. Since for each $i$ the only two vertices of the form $v_j$, having four common neighbors with $u_i$, are $v_{i+1}$ and $v_{i+m+1}$, it is now clear that $\psi = \eta$.
\end{proof}

We remark that it can be shown that the graph $\N(4;1,2,3;1)$ and the graphs from Lemma~\ref{le:fam1} are the only edge-transitive Nest graphs admitting a nontrivial automorphism fixing all the vertices $u_i$ pointwise. Since the proof is somewhat tedious, while this fact will not play a role in the remainder of our paper, we do not provide it here. Instead, we provide another infinite family of arc-transitive Nest graphs.

\begin{lemma}
\label{le:fam2}
Let $m \geq 3$ be an odd integer. Then the Nest graph $\N(4m;2,m,m+2;2m-1)$ is arc-transitive with vertex stabilizers of order $6$. 
\end{lemma}

\begin{proof}
We only give a sketch of the proof and leave the details to the reader. Lemma~\ref{le:c=a+b} implies that $\tau$ from (\ref{eq:tau}) is an automorphism of $\G = \N(4m;2,m,m+2;2m-1)$. It can be verified that the permutation $\varphi$ of the vertex set of $\G$ given by the rule 
$$
	u_i\varphi = \left\{\begin{array}{lcl}u_{-i} & ; & i \mathrm{\ even}\\
				v_{-i+1} & ; & i \mathrm{\ odd},\end{array}\right.\quad \quad
	v_i\varphi = \left\{\begin{array}{lcl}u_{-i+1} & ; & i \mathrm{\ even}\\
				v_{-i+2m+2} & ; & i \mathrm{\ odd},\end{array}\right.
$$
where $i \in \ZZ_n$,
is an automorphism of $\G$ and that the group $\langle \rho, \varphi \rangle$ has just two orbits on the set of all edges of $\G$, one of which coincides with the set of all $m$- and all $(m+2)$-spokes. The nature of the action of $\tau$ thus yields that the group $G = \langle \rho, \varphi, \tau \rangle$ acts arc-transitively on $\G$. It remains to be proved that the pointwise stabilizer of the arc $(u_0,u_1)$ is trivial. To see this we first verify that there are four $4$-cycles through $u_0u_1$ in $\G$, one through $(v_2,u_0,u_1)$, one through $(v_{m+2},u_0,u_1)$, and two through $(u_{-1},u_0,u_1)$. Any automorphism $\psi$, fixing both $u_0$ and $u_1$ must thus also fix $u_{-1}$, and then an inductive argument shows that $\psi$ fixes each $u_i$. Since the neighbors of $v_i$ of the form $u_j$ are $u_i$, $u_{i-2}$, $u_{i-m}$ and $u_{i-m-2}$ which are all fixed, $v_i$ has to be mapped to a common neighbor of these four vertices, which is easily seen to be just $v_i$, and so each $v_i$ must also be fixed by $\psi$, implying that $\psi$ is the identity.
\end{proof}

The classification of all edge-transitive Nest graphs seems to be a rather difficult problem. One could try to build on the general method that was used to classify all edge-transitive Rose Window graphs~\cite{KovKutMar10} and all edge-transitive Taba\v cjn graphs~\cite{ArrHubKutOreSpa15}. The key idea in those two classifications is to first bound the vertex stabilizers of such graphs, then apply a nice result of Lucchini from 1998 (see for instance~\cite[Theorem 2.20]{Isa_book}) to conclude that for large enough graphs the subgroup $H$, generated by the corresponding semiregular automorphism with two orbits, has nontrivial core in the full automorphism group of the graph. One then identifies the (small) graphs for which $H$ does have trivial core in the automorphism group after which all edge-transitive cyclic covers of these core-free graphs need to be classified. In addition, the possibility that the quotient graph with respect to the core of $H$ might reduce the valence of the graph needs to be considered (which of course cannot occur in the case of $5$-valent graphs). Even though this approach might still work for Nest graphs, there are some serious obstacles. First of all, unlike the case of valence $5$ there is no theoretical result giving a general bound for the order of vertex-stabilizers of a $6$-valent vertex- and edge-transitive graph while we have not been able to find an easy combinatorial argument to bound the orders of stabilizers of Nest graphs. Second, one can check that the subgroup $\langle \rho \rangle$ is core-free in the full automorphism group for all four of the graphs $\N(5;1,2,3;2)$, $\N(8;1,3,4;3)$, $\N(8;1,2,5;3)$ and $\N(12;2,4,8;5)$, while for most of the other small examples the core is of index $2$ in $\langle \rho \rangle$. Thus, if one would want to use the above mentioned method, all edge-transitive cyclic covers of the above mentioned four graphs would need to be determined as well as all edge-transitive cyclic covers of the `doubled' complete graph $K_4$ (that is the multigraph with $4$ vertices and a pair of parallel edges joining each pair of different vertices). 

Instead of trying to overcome these difficulties, we decided to refrain from attempting to classify all edge-transitive Nest graphs and to show instead that there are infinite subfamilies of such graphs (arc-transitive, as well as half-arc-transitive and with very interesting properties on their own) and to direct the focus to Question~\ref{que:main}. Furthermore, using mostly combinatorial methods, we at least classify the edge-transitive Nest graphs of girth $3$. We state the obtained classification in the following theorem. As it turns out, even obtaining the classification with this additional restriction is not trivial. Its proof is given in the next section.

\begin{theorem}
\label{the:girth3}
Let $\G$ be a Nest graph of girth $3$. Then $\G$ is edge-transitive if and only if one of the following holds:
\begin{itemize}\itemsep = 0pt
\item[(i)] $\G$ is isomorphic to one of the graphs $\N(4;1,2,3;1)$, $\N(5;1,2,3;2)$, $\N(8;1,2,5;3)$, $\N(8;1,3,4;3)$, $\N(10;1,3,4;3)$ or $\N(12;1,3,10;5)$. 
\item[(ii)] $\G \cong \N(n;1,2m+1,2m+2;1)$, where $m\geq 1$ and $n$ is an even divisor of $2(m^2+m+1)$ with $n \geq 4m+2$.
\item[(iii)] $\G \cong \N(2m;1,b,b+m+1;m-1)$, where $b = 4b_0-1$ for some $b_0 > 1$ and $m$ is a divisor of $b^2+3$ with $m \equiv 2 \pmod{4}$ and $b < 2m$.
\end{itemize}
Furthermore, $\G$ is half-arc-transitive if (iii) holds and is arc-transitive otherwise. 
\end{theorem}

\section{Edge-transitive Nest graphs of girth $3$}
\label{sec:ETNestG3}

Throughout this section let $\G = \N(n;a,b,c;k)$ be an edge-transitive Nest graph of girth $3$. Without loss of generality we may assume that $1 \leq a < b < c$. Observe that the edge-transitivity implies that each edge of $\G$ lies on the same number $\lambda = \lambda(\G)$ of $3$-cycles of $\G$. Since $n \geq 4$, each $3$-cycle of $\G$ containing the edge $u_0u_1$ consists of the edge $u_0u_1$ and two spokes, and so the number $\lambda$ is equal to the number of occurrences of the number $1$ in the list $[a, b-a, c-b, n-c]$. Thus $1 \leq \lambda \leq 4$. We first deal with the two easy cases when $\lambda \geq 3$.

\begin{lemma}
\label{le:lambda>2}
Let $\G = \N(n;a,b,c;k)$ be edge-transitive of girth $3$. Then $\lambda = 4$ if and only if $\G \cong \N(4;1,2,3;1) \cong K_{2,2,2,2} \cong K_8 - 4K_2$ (the complete graph on $8$ vertices minus a perfect matching), and $\lambda = 3$ if and only if $\G \cong \N(5;1,2,3;2)$ (the complement of the Petersen graph).
\end{lemma}

\begin{proof}
By the above remarks $a = b-a = c-b = n-c = 1$ must hold for $\lambda = 4$ to hold. It is now clear that $\G = \N(4;1,2,3;1)$ (replacing $k$ by $n-k$ if necessary). Since the graph $\N(4;1,2,3;1) \cong K_{2,2,2,2}$ is clearly edge-transitive, this proves the first part of the lemma.

Suppose now that $\lambda = 3$. By Lemma~\ref{le:isom} we can assume $a = 1$, $b = 2$ and $c = 3$ with $n \geq 5$. None of the neighbors $u_{-2}$ and $u_{-3}$ of $v_0$ is a neighbor of $u_0$ (recall that $n \geq 5$), and so the spoke $u_0v_0$ can only be a part of three $3$-cycles if both $v_k$ and $v_{-k}$ are neighbors of $u_0$ and are thus contained in $\{v_1, v_2, v_3\}$. Since $k \neq -k$ this clearly implies that $n = 5$ and $k \in \{2,3\}$. Thus $\G \cong \N(5;1,2,3;2)$, which is clearly edge-transitive being the complement of the Petersen graph.
\end{proof}

The remaining cases $\lambda = 2$ and $\lambda = 1$ require more detailed consideration. We deal with each of these cases in a separate subsection. 

\subsection{The case $\lambda = 2$}
\label{subsec:lambda=2}

Throughout this subsection we assume $\G = \N(n;a,b,c;k)$ is an edge-transitive Nest graph of girth $3$ with $\lambda = 2$. Observe that in this case for any pair of adjacent vertices $x$ and $y$ the vertex $x$ has three neighbors which are not neighbors of $y$ nor are $y$ itself, with the same holding for $y$. In some of the arguments it will prove useful to consider the subgraph of $\G$ induced on these six vertices, that is on $(\G(x) \cup \G(y))\setminus (\{x,y\} \cup (\G(x) \cap \G(y)))$. We call this subgraph the {\em local structure} with respect to the edge $xy$ and denote it by $\G_{xy}$. Moreover, the subgraph of $\G_{xy}$ induced on the three neighbors of $x$ contained in $\G_{xy}$, will be denoted by $\G_{xy}^x$.

Lemma~\ref{le:isom} and the remarks at the beginning of Section~\ref{sec:ETNestG3} imply that we have two essentially different possibilities. Either $a = 1$, $b = 2$ and $4 \leq c \leq n-2$, or $a = 1$ and $3 \leq b = c - 1 \leq n - 3$. We first prove that there is only one graph satisfying the first possibility.

\begin{lemma}
\label{le:lambda2_1}
The graph $\G = \N(n;1,2,c;k)$ is edge-transitive with $\lambda = 2$ if and only if $\G \cong \N(8;1,2,5;3)$ which is in fact arc-transitive.
\end{lemma}

\begin{proof}
It is easy to verify that $\N(8;1,2,5;3)$ is indeed arc-transitive with $\lambda = 2$. For the converse suppose that $\G = \N(n;1,2,c;k)$ is edge-transitive with $\lambda = 2$. Observe that in this case $k > 1$ since otherwise the edge $u_0v_1$ would lie on at least three $3$-cycles. Moreover, for the edge $u_0u_1$ to lie on two $3$-cycles, $3 < c < n-1$ has to hold, implying $n \geq 6$. Thus $\{c,c-1,c-2\} \cap \{1,-1\} = \emptyset$, and so for the edge $u_0v_c$ to lie on two $3$-cycles both $v_{c+k}$ and $v_{c-k}$ must be neighbors of $u_0$. It follows that $c+k$ and $c-k$ are two different members of the set $\{0,1,2\}$, implying that $2c \in \{1,2,3\}$ and $2k \in \{\pm 1, \pm 2\}$. In fact, since we can assume $k < n/2$, this implies $2k \in \{-1, -2\}$. 

A similar consideration of the possible $3$-cycles containing $u_0v_0$ reveals that precisely one of the vertices $v_k, v_{-k}$ must be adjacent to $u_0$ and is thus contained in $\{v_1, v_2, v_c\}$. Since $k > 1$ we either have $k = 2$ or $-k = c$ (recall that $k < n/2$, while since $2c \in \{1,2,3\}$ we get $c > n/2$). If $k = 2$, then the fact that $n \geq 6$ and $2k \in \{-1,-2\}$ yields $n = 6$ and $c = 4$, but then the edge $u_0v_0$ lies on three $3$-cycles, a contradiction. Thus $k \neq 2$ (implying $n \geq 7$), and so $k = -c$. Now, if $n$ is odd then $k = (n-1)/2$ and $c = (n+1)/2$, but then $k > 2$ implies that the edge $u_0v_2$ lies on just one $3$-cycle, a contradiction. It follows that $n$ is even, say $n = 2n_0$, and $k = n_0-1$, $c = n_0+1$. 
Recall that this implies $n \geq 8$ and consider the local structure $\G_{u_0v_1}$. Since $k \geq 3$ the subgraph $\G_{u_0v_1}^{u_0}$ has one vertex of valence $2$ (namely $v_{n_0+1}$) and two vertices of valence $1$ (namely $v_0$ and $v_2$). Similarly the vertex $u_{n_0}$ is the unique vertex of $\G_{u_0v_1}^{v_1}$ of valence $2$ and moreover, $u_{n_0} \sim v_{n_0+1}$. Now, edge-transitivity of $\G$ implies that the same situation must occur in the local structure $\G_{u_0u_1}$. Here the two vertices of valence $2$ in $\G_{u_0u_1}^{u_0}$ and $\G_{u_0u_1}^{u_1}$ are $v_0$ and $v_3$, respectively, and consequently $v_0 \sim v_3$. This implies $k = 3$, and so $n = 8$. Thus $\G \cong \N(8;1,2,5;3)$, as claimed.
\end{proof}

We now classify the edge-transitive Nest graphs $\N(n;a,b,c;k)$ with $a = 1$ and $3 \leq b = c-1 \leq n-3$. We first show that, except for two small exceptions, $k = 1$ must hold in this case.

\begin{lemma}
\label{le:lambda2_2}
Let $\G = \N(n;1,b,b+1;k)$ be edge-transitive with $3 \leq b \leq n-3$ and $k < n/2$. Then either $k=1$ or $\G$ is isomorphic to one of the graphs $\N(8;1,3,4;3)$ and $\N(10;1,3,4;3)$ which are both arc-transitive.
\end{lemma}

\begin{proof}
Observe that the assumptions on the parameters imply that the edge $u_0u_1$ lies on two $3$-cycles, so that $\lambda = 2$. Suppose $k > 1$ and let us prove that in this case $\G$ is isomorphic to one of the two graphs from the statement of the lemma. By Lemma~\ref{le:isom} we have $\G \cong \N(n;1,-b,-b+1)$, and so we can assume $b \leq n/2$. Thus, $b \geq 3$ implies $n \geq 6$. In fact, since the edge $v_0v_2$ lies on three $3$-cycles in the graph $\N(6;1,3,4;2)$, we have $n \geq 7$. 

Observe that $u_{-1}$ is a common neighbor of $u_0$ and $v_0$, and so $\lambda = 2$ implies that these two vertices have exactly one more common neighbor. Consequently, $|\{u_1,v_1,v_b,v_{b+1}\} \cap \{v_k,v_{-k},u_{-b}, u_{-b-1}\}| = 1$. Since $b \leq n/2$ and $n \geq 7$, we have that $1 \notin \{-b, -b-1\}$, and so $1 < k < n/2$ implies $|\{b,b+1\} \cap \{k,-k\}| = 1$. A similar argument considering the possibilities for the common neighbor of $u_0$ and $v_1$, other than $u_1$, yields $|\{b,b+1\} \cap \{1+k, 1-k\}| = 1$. 

We claim that $b = k$ must hold. It is easy to see that since $k < n/2$ and $b \leq n/2$ the only other possibility is that $n = 2m$ is even and $b = m$, $k = m-1$. To see that this is not possible consider the local structures $\G_{u_0u_1}$ and $\G_{u_0v_0}$. One can verify that there are no edges between the vertices of $\G_{u_0u_1}^{u_0}$ and $\G_{u_0u_1}^{u_1}$ (recall that $n \geq 7$). On the other hand $v_1 \in \G_{u_0v_0}^{u_0}$ is adjacent to $u_{m}\in \G_{u_0v_0}^{v_0}$, contradicting the edge-transitivity of $\G$.

We thus have $k = b$, that is $\G = \N(n;1,b,b+1;b)$. Consider the local structure $\G_{u_0u_1}$ and let $\mu$ be the number of edges between the vertices of $\G_{u_0u_1}^{u_0}$ and those of $\G_{u_0u_1}^{u_1}$. It is easy to see that the assumptions imply $\mu = 0$, except for the case when at least one of $b = 3$ (in which case $u_2 \sim v_b$ and $u_{-1} \sim v_{2}$) or $2b+2 = 0$ (in which case $v_0 \sim v_{b+2}$) holds. Thus $\mu \in \{0,1,2,3\}$. Since the vertex $u_{-b}$ of $\G_{u_0v_0}^{v_0}$ is adjacent to the vertex $v_1$ of $\G_{u_0v_0}^{u_0}$, edge-transitivity of $\G$ implies $\mu > 0$. Moreover, if $2b+2 = 0$ then the vertices $v_{-b}$ and $u_{-b-1}$ of $\G_{u_0v_0}^{v_0}$ are adjacent to $u_1$ and $v_{b+1}$ of $\G_{u_0v_0}^{u_0}$, respectively, implying that $\mu \geq 2$. This implies $b = 3$, and consequently $\mu = 3$ or $\mu = 2$, depending on whether $2b+2 = 0$ holds or not. If $2b+2 = 0$ then $n = 8$ and we get the graph $\N(8;1,3,4;3)$. If however $2b+2 \neq 0$ then the fact that there must be precisely two edges between the vertices of $\G_{u_0v_0}^{u_0}$ and $\G_{u_0v_0}^{v_0}$ implies that precisely one of the vertices $v_{-b}, u_{-b-1}$ must be adjacent to precisely one of $u_1, v_{b+1}$. Since $b = 3$, $n \geq 7$ and $2b+2 \neq 0$ this can only happen if $v_{-b} \sim v_{b+1}$ in which case $3b+1 = 0$, that is $n = 10$, and so $\G = \N(10;1,3,4;3)$. The arc-transitivity of the graphs $\N(8;1,3,4;3)$ and $\N(10;1,3,4;3)$ can easily be verified.
\end{proof}

In the next two lemmas we classify the edge-transitive Nest graphs of the form $\N(n;1,b,b+1;1)$ with $3 \leq b \leq n-3$. Recall that we can assume $b \leq n/2$. 

\begin{lemma}
\label{le:lambda2_3}
Let $\G = \N(n;1,b,b+1;1)$ be edge-transitive with $3 \leq b \leq n/2$. Then $b = 2m+1$ for some $m \geq 1$ and $n$ is an even divisor of $2(m^2+m+1)$. Moreover, $\G$ is arc-transitive and its automorphism group acts regularly on its arc-set, except for the graph $\N(6;1,3,4;1)$, which is arc-transitive with vertex stabilizers of order $12$.
\end{lemma}

\begin{proof}
By assumption, $n \geq 6$. Consider now the local structure $\G_{u_0u_1}$. Since $b \neq 1$ the subgraph $\G_{u_0u_1}^{u_0}$ is the $2$-path $(v_0,u_{-1},v_b)$ and the subgraph $\G_{u_0u_1}^{u_1}$ is the $2$-path $(v_2,u_2,v_{b+2})$. Due to the edge-transitivity of $\G$, a similar situation occurs for any pair of adjacent vertices $x$ and $y$ of $\G$, that is, if $x \sim y$ then each of $\G_{xy}^x$ and $\G_{xy}^y$ is a $2$-path. We denote the internal vertex of the $2$-path $\G_{xy}^y$ by $s(x,y)$. Thus $s(u_0,u_1) = u_2$ and $s(u_1,u_0) = u_{-1}$. It is easy to verify that for all $i \in \ZZ_n$ the following hold:
\begin{equation}
\label{eq:s_lam_2}
\begin{array}{cc}
s(u_i,u_{i+1}) = u_{i+2},\ s(u_{i},u_{i-1}) = u_{i-2}, &  s(u_i,v_i) = u_{i-b-1},\ s(v_i,u_i) = v_{i+b+1},\\
s(u_i, v_{i+1}) = u_{i+1-b},\ s(v_{i},u_{i-1}) = v_{i+b-1}, & s(u_i,v_{i+b}) = u_{i+b-1},\ s(v_i, u_{i-b}) = v_{i-b+1},\\
s(u_i,v_{i+b+1}) = u_{i+b+1},\ s(v_i,u_{i-b-1}) = v_{i-b-1}, & s(v_i,v_{i+1}) = v_{i+2},\ s(v_i,v_{i-1}) = v_{i-2}.
\end{array}
\end{equation}

Observe that this gives rise to specific natural closed walks in $\G$ defined as follows. Starting from a given arc $(x,y)$ of $\G$ we then follow the vertices $s(x,y), s(y,s(x,y))$, etc. until we eventually come back to the arc $(x, y)$. We call the corresponding closed walk the {\em $s$-walk} containing $(x, y)$. In view of (\ref{eq:s_lam_2}), the $s$-walk containing $(u_0,u_1)$ is the $n$-cycle $(u_0,u_1,u_2, \ldots , u_{n-1})$. Since $c = b+1 = a+b$, Lemma~\ref{le:c=a+b} implies that $\G$ is arc-transitive, and so all $s$-walks in $\G$ are $n$-cycles. 

Now, let $\vartheta \in \Aut(\G)$ be an automorphism mapping the arc $(u_0,u_1)$ to the arc $(u_0,v_0)$. Then $\vartheta$ maps the $n$-cycle $(u_0,u_1, \ldots , u_{n-1})$ to the $s$-walk containing $(u_0,v_0)$. In view of (\ref{eq:s_lam_2}) the latter starts with $(u_0, v_0, u_{-b-1}, v_{-b-1}, u_{2(-b-1)}, v_{2(-b-1)}, \ldots)$, and so $n$ must be even and 
$$
	u_{2i}\vartheta = u_{-i(b+1)}\quad \mathrm{and}\quad u_{2i+1}\vartheta = v_{-i(b+1)}, \ \mathrm{where}\ i \in \ZZ_n,
$$
must hold. In particular, $\langle b+1 \rangle$ is the index $2$ subgroup of $\ZZ_n$, that is $\langle b+1 \rangle = \langle 2 \rangle$. Thus $b = 2m+1$ is odd and the vertices of the form $v_i$ are mapped to the vertices of the form $u_j$ and $v_j$ with $j$ odd.

Similarly, there exists an automorphism $\xi$ of $\G$, mapping the arc $(u_0,u_1)$ to the arc $(u_0,v_1)$. A similar argument to the above shows that 
$$
	u_{2i}\xi = u_{i(1-b)}\quad \mathrm{and}\quad u_{2i+1}\xi = v_{i(1-b)+1},\ \mathrm{where\ }i \in \ZZ_n.
$$
It follows that $\langle b-1 \rangle = \langle 2 \rangle$ in $\ZZ_n$, which together with $\langle b+1 \rangle = \langle 2 \rangle$ implies that $n \equiv 2 \pmod 4$. Now, since $v_0$ is adjacent to both $u_0$ and $u_{-1}$, its image under $\vartheta$ is a common neighbor of $u_0$ and $v_{b+1}$, and so $v_0\vartheta \in \{u_1, v_b\}$. If $v_0\vartheta = u_1$, then the fact that $v_1$ is a common neighbor of $u_0, u_1$ and $v_0$ implies that $v_1\vartheta = v_1$, and so $\vartheta$ fixes the entire $s$-walk containing $(u_0,v_1)$ pointwise. But then (\ref{eq:s_lam_2}) implies that $\vartheta$ fixes each $u_j$ with $j \in \langle b-1 \rangle = \langle 2 \rangle$, and so $u_2 = u_2\vartheta = u_{-b-1}$, implying that $b+3 = 0$. As $n \leq n/2$ we get the graph $\G = \N(6;1,3,4;1)$ which is easily seen to be arc-transitive with vertex stabilizers of order $12$. Moreover, setting $m = 1$ we have $3 = b = 2m+1$ and $6 = n = 2(m^2+m+1)$.

We can thus assume that $v_0\vartheta = v_b$ and consequently $v_1\vartheta = u_{-1}$. It follows that $u_{1-b}\vartheta = s(u_0,v_1)\vartheta = s(u_0,u_{-1}) = u_{-2}$. However, since $u_{1-b} = u_{-2m}$ we also have $u_{1-b}\vartheta = u_{m(b+1)}$, and so $m(b+1)+2 = 0$, that is $2(m^2+m+1) = 0$. Thus $\G = \N(n;1,2m+1,2m+2;1)$ where $n$ is an even divisor of $2(m^2+m+1)$ and $b = 2m+1 \leq n/2$. 

For the last part of the lemma it suffices to show that unless $\G = \N(6;1,3,4;1)$, the only automorphism of $\G$ fixing the arc $(u_0,u_1)$ is the identity. Now, since any such automorphism necessarily fixes the entire $s$-walk containing $(u_0,u_1)$ pointwise, it fixes each $u_i$. But as $v_i$ is adjacent to each of the fixed vertices $u_{i}, u_{i-1}, u_{i-b}, u_{i-b-1}$, it is now clear that the only way $v_i$ could be moved is if it was mapped to $v_{i-b}$ in which case $2b = 0$ would have to hold. But then $\G = \N(6;1,3,4;1)$, as claimed. 
\end{proof}

\begin{lemma}
\label{le:lambda2_fam}
For every $m \geq 1$ and every even divisor $n$ of $2(m^2+m+1)$ with $n \geq 4m+2$ the graph $\N(n;1,2m+1,2m+2;1)$ is arc-transitive with $\lambda = 2$.
\end{lemma}

\begin{proof}
Since $b = 2m+1 \geq 3$ and $c = 2m+2 \leq n-2$ it is clear that the edge $u_0u_1$ lies on precisely two $3$-cycles of the graph $\G = \N(n;1,2m+1,2m+2;1)$. Moreover, Lemma~\ref{le:c=a+b} implies that we only need to prove that $\G$ is edge-transitive. To this end we consider the permutation $\vartheta$ of $V(\G)$, defined by the rule
$$
	u_{2i}\vartheta = u_{-i(b+1)}, \ u_{2i+1}\vartheta = v_{-i(b+1)},\ v_{2i}\vartheta = v_{b-i(b+1)},\ v_{2i+1}\vartheta = u_{-1-i(b+1)}\quad \mathrm{for\ all\ }i \in \ZZ_n.
$$
Since $\gcd(b+1,n)$ divides $\gcd(2m+2, 2(m^2+m+1)) = 2$, it is clear that $\langle b+1 \rangle = \langle 2 \rangle$ in $\ZZ_n$, and so $\vartheta$ is a bijection. To prove that it also preserves adjacency in $\G$, we consider its action on the edges of the six different types (rim, hub, spokes). Since the image of a vertex depends on the parity of its index we need to consider twelve different possibilities. 

We consider six of them and leave the remaining six to the reader. For instance, the rim edges $u_{2i}u_{2i+1}$ and $u_{2i+1}u_{2i+2}$ are mapped to the $0$-spoke $u_{i(-b-1)}v_{i(-b-1)}$ and the $(b+1)$-spoke $v_{i(-b-1)}u_{(i+1)(-b-1)}$, respectively. Similarly, the $0$-spokes $u_{2i}v_{2i}$ and $u_{2i+1}v_{2i+1}$ are mapped to the $b$-spoke $u_{i(-b-1)}v_{b+i(-b-1)}$ and the $1$-spoke $v_{i(-b-1)}u_{-1+i(-b-1)}$, respectively. Finally, since $m(b+1) = m(2m+2) = 2(m^2+m+1)-2 = -2$ in $\ZZ_n$, the $b$-spokes $u_{2i}v_{2i+b}$ and $u_{2i+1}v_{2i+1+b}$ are mapped to the rim edge $u_{i(-b-1)}u_{i(-b-1)+1}$ and the hub edge $v_{i(-b-1)}v_{i(-b-1)+1}$. Similar considerations show that $\vartheta$ also preserves adjacency for the remaining six possible types of edges, and so $\vartheta \in \Aut(\G)$. The nature of its action reveals that $\langle \rho, \vartheta\rangle$ acts transitively on the edge-set of $\G$, and so Lemma~\ref{le:c=a+b} implies that $\G$ is arc-transitive.
\end{proof}

Combining together Lemmas~\ref{le:lambda2_1}--\ref{le:lambda2_3} allows us to complete the classification of the edge-transitive Nest graphs of girth $3$ with $\lambda = 2$. 

\begin{proposition}
\label{pro:lambda2}
Let $\G = \N(n;a,b,c;k)$ be a Nest graph of girth $3$. Then $\G$ is edge-transitive with $\lambda = 2$ if and only if $\G$ is isomorphic to one of the following graphs:
\begin{itemize}
\itemsep = 0pt
\item $\N(8;1,2,5;3)$;
\item $\N(8;1,3,4;3)$;
\item $\N(10;1,3,4;3)$;
\item $\N(n;1,2m+1,2m+2;1)$, where $m\geq 1$ and $n$ is an even divisor of $2(m^2+m+1)$ with $n \geq 4m+2$.
\end{itemize} 
Moreover, all of these graphs are arc-transitive.
\end{proposition}

\subsection{The case $\lambda = 1$}
\label{subsec:lambda=1}

Throughout this subsection we assume $\G = \N(n;a,b,c;k)$ is an edge-transitive Nest graph of girth $3$ with $\lambda = 1$. In this case each edge of $\G$ lies on a unique $3$-cycle of $\G$, and so the set of $3$-cycles of $\G$ decomposes the edge-set of $\G$. The subgroup $R = \langle \rho \rangle$, where $\rho$ is as in (\ref{eq:rho}), thus has two orbits of length $n$ in its action on the set of all $3$-cycles of $\G$. 

In view of Lemma~\ref{le:isom} we can assume $a = 1$, $3 \leq b$, $b+2 \leq c \leq n-2$ and $b-1 \leq n-c$. Therefore, $b \leq (n-1)/2$. One of the two $R$-orbits of $3$-cycles thus consist of all the $3$-cycles $(u_i,u_{i+1},v_{i+1})$, $i \in \ZZ_n$. Consequently, the other $R$-orbit of $3$-cycles consists of $3$-cycles having one hub edge, a $b$-spoke and a $c$-spoke. Thus $c-b$ equals either $k$ or $-k$ in $\ZZ_n$. With no loss of generality we can assume $c-b = k$ (note however that we can now no longer assume $k < n/2$). The $R$-orbit of $3$-cycles containing a hub edge then consists of all the $3$-cycles $(u_i,v_{i+b}, v_{i+b+k})$. Using the fact that $3 \leq b \leq (n-1)/2$ and $b+k \leq n-2$ one can easily check that the assumption that $(u_{-1},u_0,v_0)$ is the only $3$-cycle containing $u_0v_0$, as well as the only $3$-cycle containing $u_{-1}v_0$, implies
\begin{equation}
\label{eq:aux1}
	2 \leq k \leq n-5 \quad \mathrm{and}\quad b \notin \{-2k, 1-2k, -k,1-k,k,k+1\}.
\end{equation}
We start with the following straightforward but useful observation.

\begin{lemma}
\label{le:rotation}
Let $\G = \N(n;1,b,b+k;k)$, where $3 \leq b \leq (n-1)/2$, $b+k \leq n-2$ and $b$ and $k$ satisfy (\ref{eq:aux1}), be edge-transitive. Then for any $3$-cycle of $\G$ there exists an automorphism of $\G$ cyclically permuting its vertices. 
\end{lemma}

\begin{proof}
Since the edge $u_0u_1$ lies on just one $3$-cycle, edge-transitivity implies that $\lambda = 1$. The claim thus clearly holds if $\G$ is arc-transitive. If however it is not arc-transitive, then edge-transitivity implies it is half-arc-transitive. Since no automorphism of a half-arc-transitive graph can interchange a pair of adjacent vertices it is thus clear that for any of the two natural orientations of the edges of $\G$, induced by the action of $\Aut(\G)$, all of the $3$-cycles are directed, and so we can again find an appropriate automorphism. 
\end{proof}

In the remainder of the paper we will say that a $2$-path is an {\em induced $2$-path} if its endvertices are not adjacent. In other words, induced $2$-paths are the $2$-paths that do not lie on a $3$-cycle.

We now identify four $R$-orbits of induced $5$-cycles of $\G$ (note that a $5$-cycle is induced if and only if all of its $2$-paths are induced). The representatives of the four $R$-orbits, whose members are said to be of type $g.1$, $g.2$, $g.3$ and $g.4$, respectively, are
$$
	(u_0,v_1,u_{1-b},u_{-b},v_0),\ (u_0,v_1,u_{1-b-k},u_{-b-k},v_0),\ (u_0, v_{b+k},u_k,v_k,v_0)\ \mathrm{and}\  (u_0, v_{b+k}, u_k, v_{k+1},v_1),
$$
respectively. Using (\ref{eq:aux1}) it is easy to verify that these $5$-cycles are all induced. The $5$-cycles of the above four $R$-orbits will be called {\em generic}. Observe that each spoke of $\G$ lies on four different generic $5$-cycles while each rim and each hub edge lies on two generic $5$-cycles. Since $\G$ is edge-transitive this implies that $\G$ must have some non-generic induced $5$-cycles as well. We can in fact prove more but before stating the next lemma we introduce some notation. For a cycle $C$ of $\G$ and consecutive vertices $x$ and $y$ on $C$ we let $s_C(x,y)$ be the successor of $y$ on $C$ when $C$ is traversed in the direction from $x$ to $y$. Thus, if $C$ is the generic $5$-cycle of type $g.1$, containing the edge $u_0u_1$, then $s_C(u_0,u_1) = v_{b+1}$.

\begin{lemma}
\label{le:five}
Let $\G$ be as in Lemma~\ref{le:rotation}. Then each edge of $\G$ lies on at least five induced $5$-cycles and each induced $2$-path of $\G$ lies on at least one induced $5$-cycle.
\end{lemma}

\begin{proof}
Observe that for the four generic $5$-cycles containing $u_0v_0$, say $C_i$, where $1 \leq i \leq 4$, the successors $s_{C_i}(u_0,v_0)$ are all different. Now, take $\psi \in \Aut(\G)$ such that it maps the arc $(u_0, v_0)$ to $(v_0, u_{-1})$ (which exists due to Lemma~\ref{le:rotation}). Since for an induced cycle $C$ and consecutive vertices $x$ and $y$ contained in $C$, $s_C(x,y)$ cannot be the common neighbor of $x$ and $y$, it follows that for each $x \in \{u_{-2}, v_{-1}, v_{b-1}, v_{b+k-1}\}$ there is an induced $5$-cycle $C$ containing the edge $v_0u_{-1}$ such that $s_C(v_0,u_{-1}) = x$. However, there is no generic $5$-cycle through the $2$-path $(u_{-2}, u_{-1},v_0)$, and so there must exist a non-generic induced $5$-cycle containing it. The edge $u_{-1}v_0$ thus lies on at least five induced $5$-cycles, proving the first claim of the lemma (recall that $\G$ is edge-transitive). Since the successors $s_{C'_i}(u_{-1},v_0)$ for the four generic $5$-cycles $C'_i$ through $u_{-1}v_0$ are also all different, this shows that each induced $2$-path, containing the edge $u_{-1}v_0$, lies on at least one induced $5$-cycle, which also proves the second claim of the lemma.
\end{proof}

In the remainder of this section a careful investigation of the induced $5$-cycles of $\G$ is undertaken. Before doing so we first prove that but for one exception the graph $\G$ contains no $4$-cycles.

\begin{lemma}
\label{le:4-cyc}
Let $\G$ be as in Lemma~\ref{le:rotation}. Then $\G$ contains no $4$-cycles unless $\G \cong \N(12;1,3,10;5)$.
\end{lemma}

\begin{proof}
Suppose $\G$ does contain $4$-cycles. Since each edge of $\G$ lies on a unique $3$-cycle all $4$-cycles are induced. Observe that a $4$-cycle contains zero, two or four spokes (we say that a $4$-cycle is a {\em zero-spoke}, a {\em two-spoke} or a {\em four-spoke} $4$-cycle, respectively), where the first option occurs if and only if $4k = 0$ in $\ZZ_n$. Let $c_4$ denote the number of $4$-cycles through any given edge of $\G$ (recall that $\G$ is edge-transitive) and for each $i \in \{0,2,4\}$ let $N_i$ be the number of $4$-cycles of $\G$ containing $i$ spokes. We now count the pairs of a $4$-cycle and an edge contained in it, and also the pairs of a $4$-cycle and a spoke contained in it. Doing this in two different ways we obtain (note that $\G$ has $4n$ spokes and $2n$ non-spokes)
$$
	6nc_4 = 4(N_0 + N_2 + N_4) \quad \mathrm{and} \quad 4nc_4 = 2N_2 + 4N_4.
$$
It follows that $N_2 + 2N_0 = nc_4 = 2N_4 - 2N_0$. Observe that the $R$-orbits of two-spoke $4$-cycles are all of length $n$, while an $R$-orbit of a four-spoke $4$-cycle is of length $n$ or $n/2$. Thus $2N_4$ is divisible by $n$, and so $nc_4 = 2N_4 - 2N_0$ implies that $N_0 = 0$ (note that if $N_0 \neq 0$ then $N_0 = n/4$). This finally proves that
\begin{equation}
\label{eq:c_4}
nc_4 = N_2 = 2N_4.
\end{equation}
In particular, $N_4 \neq 0$. We now analyze the possible four-spoke $4$-cycles. To any such $4$-cycle a {\em code} of length $4$ can be assigned in the following way. We start traversing the cycle at a vertex of the form $u_i$ and then for each traversed edge put a $0$, $1$, $b$ or $c$ in the code, depending on whether we traversed a $0$-, $1$-, $b$- or $(b+k)$-spoke, respectively. We claim that none of the $2$-paths $(u_i,v_{i+b+k},u_{i+k})$ or $(v_i,u_i,v_{i+1})$ can lie on a four-spoke $4$-cycle. Suppose to the contrary that $C$ is a four-spoke $4$-cycle containing the $2$-path $(u_i,v_{i+b+k},u_{i+k})$. Since $C$ is an induced $4$-cycle we must have $s_C(v_{i+b+k},u_i) \in \{v_i, v_{i+1}\}$ and $s_C(v_{i+b+k}, u_{i+k}) \in \{v_{i+k}, v_{i+k+1}\}$, implying that one of $i = i+k+1$ and $i+1 = i+k$ holds, which contradicts (\ref{eq:aux1}). A similar argument shows that the $2$-path $(v_i,u_i,v_{i+1})$ also cannot lie on a four-spoke $4$-cycle. This shows that in the code of any four-spoke $4$-cycle each $0$ or $1$ is followed by a $b$ or $c$ and vice versa. We can thus assume that the first symbol of the code is a $0$ or a $1$. Up to 2-step cyclic rotations the possible codes are then $(0,b,0,b), (0,b,0,c), (0,b,1,b), (0,b,1,c), (0,c,0,c), (0,c,1,b), (0,c,1,c), (1,b,1,b), (1,b,1,c)$ and $(1,c,1,c)$. Since $2b \leq n-1$ none of the codes $(0,b,0,b), (0,b,1,b)$ and $(1,b,1,b)$ is possible. In the following table we list the remaining seven possible codes, the corresponding necessary and sufficient conditions for the existence of the corresponding four-spoke $4$-cycles and the lengths of the corresponding $R$-orbits of $4$-cycles. We name the $R$-orbits by $\oo_1$-$\oo_7$.
$$
\begin{array}{|c|c|c|c|}
\hline
\mathrm{name} & \mathrm{code} & \mathrm{condition} & \mathrm{length} \\
\hline
\oo_1 & (0,b,0,c) & 2b+k = 0 & n \\ \hline
\oo_2 & (0,b,1,c) & 2b+k-1 = 0 & n \\ \hline
\oo_3 & (0,c,0,c) & 2b+2k = 0 & n/2 \\ \hline
\oo_4 & (0,c,1,b) & 2b+k-1 = 0 & n \\ \hline
\oo_5 & (0,c,1,c) & 2b+2k-1 = 0 & n \\ \hline
\oo_6 & (1,b,1,c) & 2b+k-2 = 0 & n \\ \hline
\oo_7 & (1,c,1,c) & 2b+2k-2 = 0 & n/2 \\ \hline
\end{array}
$$

Using (\ref{eq:aux1}) we find that out of the six conditions from the table the only two that can possibly hold simultaneously are the ones for $\oo_1$ and $\oo_7$. If this is indeed the case then $N_4 = 3n/2$, and so (\ref{eq:c_4}) implies $c_4 = 3$ and $N_2 = 3n$. Moreover, $n$ is even, $k = 2$ and $b = n/2 - 1$. This implies that $(u_0,u_1,u_2,v_{n/2+1})$ and $(u_0,u_1,v_2,v_0)$ are both (two-spoke) $4$-cycles, implying that their $R$-orbits, together with $\oo_1$ and $\oo_7$, provide the three $4$-cycles through each rim edge, each $0$-spoke and each $(b+k)$-spoke, while they only provide two $4$-cycles through each $1$-spoke and each $b$-spoke, and only one $4$-cycle through each hub edge. The remaining $R$-orbit of (two-spoke) $4$-cycles thus consists of $4$-cycles containing two consecutive hub edges, a $1$-spoke and a $b$-spoke. Since the $4$-cycles must be induced, they are of the form $(u_i,v_{i+1},v_{i+3},v_{i+5})$ with $n/2 - 1 = b = 5$. But then $\G = \N(12;1,5,7;2)$, which is not edge-transitive, as can easily be verified. It follows that precisely one of the conditions from the above table holds, implying that $N_4 \in \{n/2, n, 2n\}$. We now analyze each of these three possibilities. 
\medskip

\noindent
Case 1: $N_4 = n/2$.\\
In this case one of $2b+2k = 0$ and $2b+2k-2 = 0$ holds, and so (\ref{eq:aux1}) implies that $n$ is even and one of $b+k = n/2$ and $b+k = n/2 + 1$ holds. Moreover, (\ref{eq:c_4}) implies $c_4 = 1$ and $N_2 = n$, so that a unique $R$-orbit of two-spoke $4$-cycles exists. If $b+k = n/2$, then the fact that the $4$-cycles of $\oo_3$ only contain $0$- and $(b+k)$-spokes implies that the $4$-cycle containing $u_0u_1$ must be $(v_{b}, u_0,u_{1},v_{2})$ with $b = k+2$ (recall that it is induced). Thus $2k+2 = n/2$, implying that $n = 4m$ for some integer $m$. Then $k = m-1$ and $b = m+1$, so that $\G = \N(4m;1,m+1,2m;m-1)$. If however $b+k = n/2 + 1$, a similar argument shows that $n = 4m$, $b = m$ and $k = m+1$, so that $\G = \N(4m;1,m,2m+1;m+1)$. However, as we now show, none of these graphs is edge-transitive. Suppose that $\G = \N(4m;1,m+1,2m;m-1)$ where $m \geq 3$ (recall that (\ref{eq:aux1}) holds). Let $\psi \in \Aut(\G)$ be an automorphism mapping the arc $(u_{-1},u_0)$ to $(v_1,u_0)$ (which exists by Lemma~\ref{le:rotation} since we have $\rho$ from (\ref{eq:rho})). Then $\psi$ fixes setwise the $4$-cycle $(v_m,u_{-1},u_0,v_1)$, and so it fixes $v_{m}$ and interchanges $u_{-1}$ with $v_1$. It follows that it also interchanges the common neighbor $v_0$ of $u_{-1}$ and $u_0$ with the common neighbor $u_1$ of $v_1$ and $u_0$. This implies that $\psi$ interchanges the unique $4$-cycle $(u_0,v_0,u_{2m},v_{2m})$ through $u_0v_0$ with the unique $4$-cycle $(u_0,u_1,v_2,v_{m+1})$ through $u_0u_1$. In particular, it interchanges $v_{2m}$ with $v_{m+1}$ and $u_{2m}$ with $v_2$. Moreover, it interchanges the unique $4$-cycle $(u_1,v_1,u_{2m+1},v_{2m+1})$ through $u_1v_1$ with the unique $4$-cycle $(v_0,u_{-1},u_{-2},v_{m-1})$ through $v_0u_{-1}$, and so it interchanges $u_{2m+1}$ with $u_{-2}$ and $v_{2m+1}$ with $v_ {m-1}$. But then the common neighbor $u_{2m}$ of $u_{2m+1}$ and $v_{2m+1}$ must be mapped to $v_{2m-2}$, implying that $2m-2 = 2$, which contradicts $m \geq 3$.
A similar contradiction can be obtained if $\G = \N(4m;1,m,2m+1,m+1)$.
\medskip

\noindent
Case 2: $N_4 = n$.\\
In this case one of $2b+k = 0$, $2b+2k-1 = 0$ and $2b+k-2 = 0$ holds and (\ref{eq:c_4}) implies $c_4 = 2$ and $N_2 = 2n$, so that there are two $R$-orbits of two-spoke $4$-cycles. The arguments in each of the three possibilities are similar. Suppose first that $2b+k = 0$, so that $4$-cycles from $\oo_1$ exist. Then for each $i \in \ZZ_n$ the two $4$-cycles of $\G$ containing $u_iv_i$ are in $\oo_1$. The two $4$-cycles $C$ and $C'$ through $u_1v_1$ have $s_C(u_1,v_1) \neq s_{C'}(u_1,v_1)$ and $s_C(v_1,u_1) \sim s_{C'}(v_1,u_1)$. Letting $\psi \in \Aut(\G)$ be an automorphism mapping $(u_1,v_1)$ to $(u_0,u_1)$ we thus have that $s_{C\psi}(u_0,u_1) \neq s_{C'\psi}(u_0,u_1)$ and $s_{C\psi}(u_1,u_0) \sim s_{C'\psi}(u_1,u_0)$. As there is no $4$-cycle containing $(v_0,u_0,u_1)$ (since both that contain $u_0v_0$ are in $\oo_1$) we have $\{s_{C\psi}(u_1,u_0), s_{C'\psi}(u_1,u_0)\} = \{v_b, v_{b+k}\}$, and so $C\psi$ and $C'\psi$ belong to the two different $R$-orbits of two-spoke $4$-cycles. It follows that the $4$-cycles from each of these two orbits contain one $1$-spoke and one hub edge each (recall that $c_4 = 2$). But then $s_{C\psi}(u_0,u_1) = v_2 = s_{C'\psi}(u_0,u_1)$, a contradiction. For the possibility $2b+2k-1 = 0$ the two $4$-cycles of $\G$ containing $u_iv_{i+b+k}$ are both in $\oo_5$ while for the possibility $2b+k-2 = 0$ the two $4$-cycles of $\G$ containing $u_iv_{i+1}$ are both in $\oo_6$. In both cases a similar argument as above applies. We leave the details to the reader. 
\medskip

\noindent
Case 3: $N_4 = 2n$.\\
In this case $2b+k-1 = 0$ holds, so that we have the four-spoke $4$-cycles from $R$-orbits $\oo_2$ and $\oo_4$, while (\ref{eq:c_4}) implies $c_4 = 4$ and $N_2 = 4n$. Let $C_2 \in \oo_2$ and $C_4 \in \oo_4$ be the four-spoke $4$-cycles through $u_{-1}v_0$ and let $\psi \in \Aut(\G)$ be an automorphism mapping $(u_{-1},v_0)$ to $(v_0,u_0)$. Observe that $s_{C_2}(u_{-1},v_0) \nsim s_{C_4}(u_{-1},v_0)$, while $s_{C_2}(v_0,u_{-1}) \sim s_{C_4}(v_0,u_{-1})$. Therefore, $s_{C_2\psi}(u_0,v_0) \sim s_{C_4\psi}(u_0,v_0)$ holds. Note however, that for the $4$-cycles $C_2' \in \oo_2$ and $C_4'\in \oo_4$ through $u_0v_0$ we have $s_{C_2'}(u_0,v_0) \nsim s_{C_4'}(u_0,v_0)$, showing that there must be a two-spoke $4$-cycle $C$ through $u_0v_0$ so that $s_C(u_0,v_0) \in \{v_{k}, v_{-k}\}$. Now, observe that since the rim and hub edges only lie on the two-spoke $4$-cycles which either contain one rim and one hub edge or two rim edges or two hub edges it follows that a $4$-cycle with two consecutive rim edges exists if and only if a $4$-cycle with two consecutive hub edges exists. 

If there is no $4$-cycle with two rim edges then the spokes and non-spokes alternate on each two-spoke $4$-cycle, and so for the two two-spoke $4$-cycles through $u_0v_0$ (recall that $c_4 = 4$), which we denote by $C'$ and $C''$, we have $s_{C'}(v_0,u_0) = s_{C''}(v_0,u_0) = u_1$ (recall that the $4$-cycles are induced), which is not adjacent to any of $s_{C_2'}(v_0,u_0)$ or $s_{C_4'}(v_0,u_0)$. Applying $\psi{-1}$ to $C'$ and $C''$ we thus see that since $s_{C_2}(u_{-1},v_0) \nsim s_{C_4\psi}(u_{-1}, v_0)$ one of $s_{C'\psi^{-1}}(u_{-1},v_0)$ and $s_{C''\psi^{-1}}(u_{-1},v_0)$ must coincide with one of $s_{C_2}(u_{-1},v_0)$ and $s_{C_4}(u_{-1}, v_0)$, implying there is another $4$-cycle (not from $\oo_2$ or $\oo_4$) containing two consecutive spokes, a contradiction. 

This shows that a $4$-cycle $C$ containing $(u_0,u_1,u_2)$ must exist. Since $C$ is induced, it follows that $s_C(u_1,u_0) \in \{v_0,v_b,v_{b+k}\}$ and $s_C(u_1,u_2) \in \{v_3,v_{b+2}, v_{b+k+2}$. In view of (\ref{eq:aux1}) it thus follows that one of $b+k+2 = 0$, $b = 3$ and $k = 2$ holds. Since $2b+k-1 = 0$ the first two conditions are equivalent and so by (\ref{eq:aux1}) precisely one of $b = 3$ and $k = 2$ holds. Suppose first that $k = 2$, in which case $2b + 1 = 0$, forcing $n = 2b+1$ and $\G = \N(2b+1; 1, b, b+2; 2)$. Then $C = (u_0,u_1,u_2,b_{b+2})$, $C\rho^{-1}$ and $(u_0,u_1,v_2,v_0)$ are $4$-cycles of $\G$ through $u_0u_1$, and so the fourth one, say $C'$, must also contain one hub edge. Then $s_{C'}(u_1,u_0) \in \{v_0,v_b,v_{b+2}\}$ and $s_{C'}(u_0,u_1) \in \{v_2,v_{b+1},v_{b+3}\}$ and then $b \geq 4$ implies that $b = 4$ with $n = 9$ must hold, which however is also not possible since we then have five $4$-cycles through $u_0u_1$.  

This leaves us with the possibility that $k > 2$ and $b = 3$, which then implies $k+5 = 0$. Then $C = (u_0,u_1,u_2,v_3)$ and $C' = (u_0,u_1,u_2,v_0)$ are both $4$-cycles of $\G$. Now, $s_{C}(u_2,u_1) = s_{C'}(u_2,u_1) = u_0$ is a neighbor of $v_1 = s_{C'\rho}(u_2,u_1)$, and so applying $\rho^{-1}\psi^{-1}$ we find that the fourth $4$-cycle $C''$ through $u_0v_0$ (other than $C_2'$, $C_4'$ and $C'$) has $s_{C''}(v_0,u_0) \in \{v_b, v_{b+k}\}$. As also $s_{C}(u_0,u_1) = s_{C'}(u_0,u_1) = u_2$ is a neighbor of $v_2 = s_{C\rho^{-1}}(u_0,u_1) = v_2$, while in addition $s_{C_4'}(u_0,v_0) = s_{C'}(u_0,v_0) = u_2$ is not a neighbor of $s_{C_2'}(u_0,v_0)$, applying $\rho^{-1}\psi^{-1}$ shows that $s_{C''}(u_0,v_0) = v_{-k} = v_5$, forcing $C'' = (u_0,v_0,v_5,v_{10})$ with $b+k = 10$, that is $n = 12$. Therefore, $\G = \N(12;1,3,10;7) \cong \N(12;1,3,10;5)$, as claimed.
\end{proof}

For the rest of this section we assume $\G$ is not isomorphic to $\N(12;1,3,10;5)$, so that $\G$ has no $4$-cycles by Lemma~\ref{le:4-cyc}. It follows that all $5$-cycles of $\G$ are induced, that is, any $2$-path along a $5$-cycle is induced. We now investigate the $5$-cycles of $\G$. Before making the first few useful observations we fix some notation. We let $c_5$ be the number of $5$-cycles through any given edge of $\G$ and for each $i \in \{0,2,4\}$ we let $N_i$ be the number of $5$-cycles of $\G$ with $i$ spokes.

\begin{lemma}
\label{le:2AT}
Let $\G$ be as in Lemma~\ref{le:rotation} with $\G \ncong \N(12;1,3,10;5)$. Then the following hold.
\begin{itemize}\itemsep = 0pt
\item[(i)] No $3$-path of $\G$ is contained on more than one $5$-cycle.
\item[(ii)] $N_4 = 2N_2 + 5N_0$.
\item[(iii)] $c_5 \in \{5,10\}$, unless $5k = 0$ in $\ZZ_n$, in which case $c_5 \in \{6,11\}$.
\item[(iv)] $\Aut(\G)$ has more than one orbit on the set of all induced $2$-paths of $\G$.
\end{itemize}
\end{lemma}

\begin{proof}
By Lemma~\ref{le:4-cyc} the graph $\G$ has no $4$-cycles and so the first claim is clear. Counting the number of pairs of an edge and a $5$-cycle containing it, and the number of pairs of a spoke and a $5$-cycle containing it, respectively, we get
\begin{equation}
\label{eq:c_5}
6nc_5 = 5(N_4 + N_2 + N_0)\quad \mathrm{and} \quad 4nc_5 = 4N_4 + 2N_2,
\end{equation}
which proves claim (ii). Now, observe that $N_0 \neq 0$ if and only if $5k = 0$ in $\ZZ_n$, in which case $N_0 = n/5$, while each $R$-orbit of a two-spoke $5$-cycle, as well as of a four-spoke $5$-cycle, is clearly of length $n$. Thus, if $5k \neq 0$, then the right hand side of the first equation in (\ref{eq:c_5}) is divisible by $5n$, and so $5$ divides $c_5$. If however $5k = 0$, then $6c_5 \equiv 1 \pmod{5}$, and so $c_5 \equiv 1 \pmod{5}$. 

Since all $5$-cycles are induced, (i) implies that there can be at most four $5$-cycles through any given induced $2$-path, and so $c_5 \leq 16$. Let us inspect the possible $5$-cycles through $(u_0,v_{b+k},u_k)$, where we let $C_3$ and $C_4$ be the corresponding generic $5$-cycles of types $g.3$ and $g.4$, respectively. Since $s_{C_3}(v_{b+k},u_0) = v_0$, $s_{C_4}(v_{b+k},u_0) = v_1$, $s_{C_3}(v_{b+k},u_k) = v_k$ and $s_{C_4}(v_{b+k},u_k) = v_{k+1}$ it is clear that the only way another $5$-cycle $C$ through $(u_0,v_{b+k},u_k)$ can exist is if $s_C(v_{b+k},u_0) \in \{u_{-1}, u_1\}$ and $s_C(v_{b+k},u_k) \in \{u_{k-1}, u_{k+1}\}$. In view of (\ref{eq:aux1}) the only possibility is that $k = 3$ and $C = (u_2,u_1,u_0,v_{b+k},u_k)$. In particular, the $2$-path $(u_0,v_{b+k},u_k)$ lies on two $5$-cycles, unless $k = 3$ in which case it lies on three $5$-cycles. This shows that $c_5 < 16$. A similar argument shows that the $2$-path $(v_0,u_0,v_1)$ also lies on two $5$-cycles, unless one of $3k - 1 = 0$ and $3k + 1 = 0$ holds, in which case it lies on three $5$-cycles. Thus, for $c_5 = 15$ to hold both $k = 3$ and one of $3k = 1$ and $3k = -1$ would need to hold, which by (\ref{eq:aux1}) is not possible (recall that $b \leq (n-1)/2$). Thus $c_5 < 15$, establishing claim (iii).

Finally, if $\Aut(\G)$ was transitive on the set of all induced $2$-paths of $\G$ then for any given edge $xy$ of $\G$ and the four neighbors $w_i$, $1 \leq i \leq 4$, of $y$, not adjacent to $x$, the number of $5$-cycles through $(x,y,w_i)$ would be the same, and so $c_5$ would be divisible by $4$. In view of (iii) this is not possible, and so (iv) holds.
\end{proof}

The next lemma will play a central role in the rest of this section. It shows that for any $2$-path, having the internal vertex of the form $v_i$ and consisting of a $b$- and a $(b+k)$-spoke, there are precisely two $5$-cycles through it and none of them contains a rim edge.

\begin{lemma}
\label{le:k=3}
Let $\G$ be as in Lemma~\ref{le:2AT}. Then $k > 3$ and the $2$-path $(u_0,v_{b+k},u_k)$ lies on precisely two $5$-cycles, namely the generic ones of types $g.3$ and $g.4$.
\end{lemma}

\begin{proof}
By way of contradiction suppose $k = 3$. By Lemma~\ref{le:4-cyc} the graph $\G$ has no $4$-cycles, and so none of the conditions from the table in the proof of Lemma~\ref{le:4-cyc} can hold, implying that $2b \leq n-7$. It is not difficult to verify that $b \geq 8$ also has to hold (note that (\ref{eq:aux1}) implies $b \geq 5$ while for instance if $b = 7$ the $4$-cycle $(u_0,v_7,v_4,v_1)$ exists). It follows that $8 \leq b \leq (n-7)/2$. Thus $n \geq 23$, and consequently $5k \neq 0$ in $\ZZ_n$. By Lemma~\ref{le:five} the $2$-path $(v_0,v_3,v_6)$ lies on a $5$-cycle. However, one can verify that $8 \leq b \leq (n-7)/2$ implies that this can only happen if $8 \leq b \leq 10$. One can now verify that the edge $u_0v_b$ lies on precisely six different $5$-cycles (four generic, the $5$-cycle $(u_0,v_b,u_{-3},u_{-2},u_{-1})$ and one of $(u_0,v_b,v_{b-3},v_{b-6},v_{b-9})$ and $(u_0,v_b,v_{b-3},v_{b-6},u_1)$, depending on whether $b \in \{9,10\}$ or $b=8$), contradicting Lemma~\ref{le:2AT}. Thus $k > 3$ and then the argument from the proof of Lemma~\ref{le:2AT} proves that $(u_0,v_{b+k},u_k)$ lies on precisely two $5$-cycles. 
\end{proof}

We next prove that the automorphism group of $\G$ is as small as possible in the sense that it acts regularly on the set of its arcs or edges, depending on whether $\G$ is arc-transitive or half-arc-transitive, respectively. Before proving this we make the following notational convention. Let $A = \Aut(\G)$. For a vertex $x$ of $\G$ the induced action of the stabilizer $A_x$ on the neighborhood $\G(x)$ is isomorphic to a subgroup of the group $H = \langle (1\,2), (1\,3\,5)(2\,4\,6), (1\,3)(2\,4)\rangle$, where the pairs $\{1,2\}$, $\{3,4\}$ and $\{5,6\}$ represent adjacent pairs of vertices. In the case that $\G$ is half-arc-transitive any of the two $A$-induced orientations of the edges of $\G$ is such that the $3$-cycles of $\G$ are directed, and so it is clear that the induced action of $A_x$ on $\G(x)$ is isomorphic to one of the subgroups $\langle (1\,3\,5)(2\,4\,6)\rangle$ and $\langle (1\,3)(2\,4), (3\,5)(4\,6)\rangle$. Suppose now that $\G$ is arc-transitive. By Lemma~\ref{le:2AT} the action of $A$ on the set of induced $2$-paths of $\G$ is not transitive, and so the induced action of $A_x$ on $\G(x)$ is isomorphic to one of the subgroups $\langle (1\,3\,5\,2\,4\,6) \rangle$, $\langle (1\,2)(3\,6)(4\,5), (1\,3\,5)(2\,4\,6) \rangle$ and $\langle (1\,3\,5\,2\,4\,6), (3\,6)(4\,5)\rangle$.

\begin{lemma}
\label{le:regular}
Let $\G$ be as in Lemma~\ref{le:2AT}. Then the stabilizer of any arc of $\G$ is trivial in $\Aut(\G)$. In other words, either $\G$ is arc-transitive and the action of $\Aut(\G)$ on the set of the arcs of $\G$ is regular or $\G$ is half-arc-transitive and the action of $\Aut(\G)$ on the set of the edges of $\G$ is regular.
\end{lemma}

\begin{proof}
Denote $A = \Aut(\G)$ and suppose the claim of this lemma does not hold. Then the above remarks imply that for each pair of adjacent vertices $x$ and $y$ of $\G$ there exists an element $\psi \in A$ fixing both $x$ and $y$ but interchanging $\{w_1,w_2\}$ with $\{z_1,z_2\}$, where $(y,w_1, w_2)$ and $(y,z_1,z_2)$ are $3$-cycles. Moreover, the restriction of $\psi^2$ to $\G(y)$ is the identity. 

Now, let $x = u_0$ and $y = v_{b+k}$ and let $\psi \in A$ be an automorphism as described in the previous paragraph. Since the unique common neighbor of $x$ and $y$ is $v_b$ we have $v_b \psi = v_b$. Lemma~\ref{le:k=3} implies that there are exactly two $5$-cycles through $(u_0,v_{b+k},u_k)$, namely the generic $5$-cycles $C_3 = (u_0,v_{b+k},u_k,v_k,v_0)$ (of type $g.3$) and $C_4 = (u_0,v_{b+k},u_k,v_{k+1},v_1)$ (of type $g.4$). By assumption, $\psi$ maps the adjacent pair $\{u_k,v_{b+2k}\}$ to the adjacent pair $\{u_{b+k}, u_{b+k-1}\}$ and vice versa. We now analyze the two possibilities regarding $u_k\psi$.
\medskip

\noindent
Case 1: $u_k\psi = u_{b+k}$.\\
Since $C_2 = (u_0,v_{b+k},u_{b+k},v_{b+k+1},u_1)$ is a $5$-cycle through $(u_0,v_{b+k},u_{b+k})$, it follows that either $C_3\psi = C_2$ or $C_4\psi = C_2$ holds. Thus one of $v_0\psi$ and $v_1\psi$ is $u_1$. It follows that $\psi$ fixes $u_0$ and its neighbor $v_{b+k}$, but does not fix all of the neighbors of $u_0$, and so the remarks preceding this lemma imply that $\psi$ interchanges the adjacent pairs $\{u_{-1}, v_0\}$ and $\{v_1,u_1\}$. This implies that $v_1\psi \neq u_1$, and so $\psi$ interchanges $v_0$ with $u_1$ (and consequently $u_{-1}$ with $v_1$). Then $C_3\psi = C_2$, and so $v_k\psi = v_{b+k+1}$. The common neighbor $u_{-b}$ of $v_0$ and $v_k$ is thus mapped by $\psi$ to $v_{b+1}$. To determine the image $C_4\psi$ we need to determine the image $v_{k+1}\psi$, which must be a neighbor of $u_{b+k}$. Now, since $v_{k+1}$ is not the common neighbor of $v_k$ and $u_k$, we have $v_{k+1}\psi \neq u_{b+k+1}$, and so the fact that $C_4\psi$ is induced implies that $v_{k+1}\psi \in \{v_{2b+k}, v_{2b+2k}\}$. Since $C_4\psi$ contains the rim edge $u_{-1}u_0$, Lemma~\ref{le:k=3} then implies that the $\psi$-image of the $2$-path $(u_k,v_{k+1},v_1)$ of $C_4$, that is $(u_{b+k}, v_{k+1}\psi, u_{-1})$, does not consist of a $b$- and a $(b+k)$-spoke, and so the edge $v_{k+1}v_1$ has to be mapped to a $0$-spoke (it cannot be a $1$-spoke since we already know that $v_0 = u_1\psi$), that is $v_{k+1}\psi = v_{-1}$. This finally proves that $\psi$ maps the common neighbor $u_{1-b}$ of $v_1$ and $v_{k+1}$ to the common neighbor $u_{-2}$ of $u_{-1}$ and $v_{-1}$. But then $u_{-2}$ is adjacent to $v_{b+1} = u_{-b}\psi$ and it has to be via a $0$-spoke (recall that $k \neq 3$), so that $b + 3 = 0$, contradicting $b \leq (n-1)/2$.
\medskip

\noindent
Case 2: $u_k\psi = u_{b+k-1}$.\\
Most of the argument is very similar to the one in the previous case so we leave some details to the reader. Since $C'_2 = (u_0,v_{b+k},u_{b+k-1},v_{b+k-1},u_{-1})$ is a $5$-cycle through $(u_0,v_{b+k},u_{b+k-1})$ we deduce that $C_4\psi = C'_2$, and so $\psi$ interchanges $v_1$ with $u_{-1}$ and $v_0$ with $u_1$, and maps $v_{k+1}$ of $C_4$ to $v_{b+k-1}$ of $C'_2$. Since $v_{k}$ is not the common neighbor of $u_k$ and $v_{k+1}$, it follows that $v_k\psi \in \{v_{2b+k-1}, v_{2b+2k-1}\}$, implying that the edge $v_0v_k$ cannot be mapped to a $b$- or a $(b+k)$-spoke. It is thus a $1$-spoke, and so $v_k\psi = v_2$. The common neighbor $u_{-b}$ of $v_0$ and $v_k$ is thus mapped to $u_2$, while the common neighbor $u_{-b+1}$ of $v_1$ and $v_{k+1}$ is mapped to $v_{b-1}$. Since $k \leq n-5$ the edge $u_2v_{b-1}$ must be a $1$-spoke, and so $b = 4$. Then the common neighbor $v_{-b+1}$ of $u_{-b}$ and $u_{-b+1}$ is mapped to $u_3$, which is a neighbor of the fixed vertex $v_b = v_4$. Since $v_{b+k}$ is also fixed, it follows that $v_{-b+1} = v_{b-k}$, and so $k = 7$. Moreover, the remarks preceding this lemma imply that $\psi$ in fact interchanges $v_{-b+1} = v_{-3}$ with $u_3$. But then the adjacent vertices $u_3$ and $v_7$ are mapped to $v_{-3}$ and $v_2$, respectively (recall that $v_7 = v_k$), implying that $-3-7 = 2$ in $\ZZ_n$, that is $n = 12$. However, as $b+k = 11$, this is impossible.
\end{proof}

\begin{lemma}
\label{le:HAT}
Let $\G$ be as in Lemma~\ref{le:2AT}. Then $\G$ is half-arc-transitive and the stabilizer of any vertex of $\G$ is of order $3$. 
\end{lemma}

\begin{proof}
In view of Lemma~\ref{le:regular} we only need to prove that $\G$ is half-arc-transitive. By way of contradiction suppose $\G$ is arc-transitive. The remarks preceding Lemma~\ref{le:regular} and the fact that by Lemma~\ref{le:2AT} the action of $\Aut(\G)$ on the set of induced $2$-paths of $\G$ is not transitive imply that for each pair of adjacent vertices $x$ and $y$ and a pair of adjacent neighbors $w_1$ and $w_2$ of $y$, different from $x$, the induced $2$-paths $(x,y,w_1)$ and $(x,y,w_2)$ are not in the same $\Aut(\G)$-orbit.

By Lemma~\ref{le:regular} the action of $\Aut(\G)$ on the set of arcs of $\G$ is regular, and so there is a unique automorphism $\alpha$ of $\G$, fixing $u_0$ and mapping $v_b$ to $v_{b+k}$. Thus $v_{b+k}\alpha = v_b$, and so $\alpha$ is an involution. Since $\rho$ (from (\ref{eq:rho})) is semiregular, we have that $\alpha \notin \langle \rho \rangle$. Observe first that $u_k\alpha \neq v_{b-k}$, since otherwise the two paths $(u_k,v_{k+b},u_0)$ and $(u_k,v_{b+k},v_b) = (u_0,v_b,v_{b-k})\rho^k$ would be in the same $\Aut(\G)$-orbit. We now analyze the other three possibilities for $u_k\alpha$. In the analysis we will be working with the $5$-cycles $C_3 = (u_0,v_{b+k},u_k,v_k,v_0)$ and $C_4 = (u_0,v_{b+k},u_k,v_{k+1},v_1)$, which by Lemma~\ref{le:k=3} are the only $5$-cycles through $(u_0,v_{b+k},u_k)$. 

Suppose first that $u_k\alpha = u_{-k}$. Then $C_3\alpha$ and $C_4\alpha$ are the two generic $5$-cycles through $(u_0,v_b,u_{-k})$, and since $\alpha \neq \rho^{-2k}$, the arc $(u_k,v_k)$ has to be mapped to $(u_{-k},v_{-k+1})$. Thus $\alpha$ interchanges $v_{k+1}$ with $v_{-k}$, $v_k$ with $v_{-k+1}$ and $v_0$ with $v_1$, and thus also the common neighbor $u_1$ of $u_0$ and $v_1$ with $u_{-1}$, and similarly $u_{k-1}$ with $u_{-k+1}$ and $u_{k+1}$ with $u_{-k-1}$. But then $u_0\alpha\rho\alpha = u_{-1}$ and $u_1\alpha\rho\alpha = u_0$, so that $\alpha\rho\alpha = \rho^{-1}$. Thus $v_{b+k} = v_b\alpha = v_0\rho^b\alpha = v_0\alpha\rho^{-b} = v_{1-b}$, implying that $2b+k-1 = 0$, which contradicts Lemma~\ref{le:4-cyc}.

Suppose next that $u_k\alpha = u_b$ and let $C_1 = (u_0,v_b,u_b,v_{b+1},u_1)$ be the generic $5$-cycle of type $g.1$ through $(u_0,v_b,u_b)$. Then $C_1 = C_3\alpha$ must hold, since in the case of $C_1 = C_4\alpha$ the arc $(u_k,v_{k+1})$ would be mapped by $\alpha$ to the arc $(u_b,v_{b+1})$, and so regularity of the action of $\Aut(\G)$ on the set of the arcs of $\G$ would imply $\alpha = \rho^{b-k}$, which is not the case. It follows that $v_k\alpha = v_{b+1}$ and $v_0\alpha = u_1$, and thus $\alpha$ also interchanges $v_1$ with $u_{-1}$ and $u_{k-1}$ with $u_{b+1}$. Now, observe that since $v_{k+1}$ is a neighbor of $u_k$ but is not adjacent to $v_{b+k}$, we have that $v_{k+1}\alpha \in \{v_{2b}, v_{2b+k}\}$. But if $v_{k+1}\alpha = v_{2b+k}$ then the $2$-path $(v_{b},u_b,v_{2b+k})$ is in the same $\Aut(\G)$-orbit as $(v_{k+1},u_k,v_{b+k})$, which is in the same $\Aut(\G)$-orbit as $(u_{b-1},u_b,v_{2b+k})$ (apply $\rho^{-k}\alpha\rho^b$), a contradiction. Thus $\alpha$ interchanges $v_{k+1}$ with $v_{2b}$, and so $C_4\alpha = (u_0,v_b,u_b,v_{2b},u_{-1})$ is a $5$-cycle containing a rim edge, so that the edge $u_{-1}v_{2b}$ must be a $0$-spoke (confront Lemma~\ref{le:k=3}). Thus $2b+1 = 0$. But now the $2$-path $(u_0,v_b,u_{b})$, which is in the same $\Aut(\G)$-orbit as $(u_0,v_{b+k},u_k)$ is contained on at least three $5$-cycles ($C_3\alpha$, $C_4\alpha$ and $C_4\alpha\rho^{-b}$), contradicting Lemma~\ref{le:k=3}.

Suppose finally that $u_k\alpha = u_{b-1}$ and let $C'_1 = (u_0,v_b,u_{b-1},v_{b-1},u_{-1})$ be the generic $5$-cycle of type $g.1$ through $(u_0,v_b,u_{b-1})$. The argument is very similar to the one in the previous paragraphs, so we omit some details. We first find that $\alpha$ interchanges $v_{b+2k}$ with $u_b$ and that $C'_1 = C_4\alpha$, implying that $\alpha$ interchanges $v_{k+1}$ with $v_{b-1}$, $u_{k+1}$ with $u_{b-2}$, $v_1$ with $u_{-1}$ and $v_0$ with $u_1$. Moreover, $v_k\alpha \in \{v_{2b-1},v_{2b+k-1}\}$ and $\alpha$ maps the edge $v_0v_{k}$ from $C_3$ to the $1$-spoke $u_1v_2$, so that either $2b-1 = 2$ or $2b+k-1 = 2$. Since the former contradicts $b \leq (n-1)/2$, we have that $\alpha$ interchanges $v_k$ with $v_{2b+k-1}$, and thus also $u_{k-1}$ with $v_{2b-1}$. But then $\alpha\rho^k\alpha\rho^{1-b}$ maps the $2$-path $(v_1,u_0,v_{b+k})$ to the $2$-path $(v_b,u_0,v_1)$, a contradiction.
\end{proof}

We are now finally ready to classify the edge-transitive Nest graphs with $\lambda = 1$.

\begin{proposition}
\label{pro:lambda1}
Let $\G = \N(n;a,b,c;k)$ be a Nest graph. Then $\G$ is edge-transitive of girth $3$ with $\lambda = 1$ if and only if $\G$ is isomorphic to the graph $\N(2m;1,b,b+m+1;m-1)$, where $b = 4b_0-1$ for some $b_0 \geq 1$ and $m > 2$ is an even divisor of $b^2+3$ with $m \equiv 2 \pmod{4}$ and $b < 2m$. Moreover, $\G$ is half-arc-transitive with the vertex stabilizers of order $3$ except for the graph $\N(12;1,3,10;5)$ which is arc-transitive with vertex stabilizers of order $6$.
\end{proposition}

\begin{proof}
The last part follows from Lemma~\ref{le:HAT} (the fact that the graph $\N(12;1,3,10;5)$ is indeed arc-transitive with vertex stabilizers of order $6$ can easily be verified), so we only need to prove the first part. 

Suppose first that $\G$ is edge-transitive of girth $3$ with $\lambda = 1$. By Lemma~\ref{le:HAT} the graph $\G$ is either isomorphic to $\N(12;1,3,10;5)$, in which case we can take $b_0 = 1$ and $m = 6$, or is half-arc-transitive with vertex stabilizers of order $3$. We can thus assume that the latter holds. Fix the $\Aut(\G)$-induced orientation of the edges of $\G$ in which $u_0 \to u_1$. Now, if $u_0 \ot v_b$ holds then Lemma~\ref{le:rotation} implies that $u_0 \to v_{b+k}$, and so setting $b' = b+k$ and $k' = n-k$ we have that $\G = \N(n;1,b',b'+k';k')$ with $u_0 \to v_{b'}$. With no loss of generality we can thus assume that $u_0 \to v_b$ holds. Note however, that we can now no longer assume that $b \leq (n-1)/2$ nor that $b+k < n$. Using Lemma~\ref{le:rotation} and the action of the automorphism $\rho$ from (\ref{eq:rho}) we thus find that for each $i \in \ZZ_n$
\begin{equation}
\label{eq:orientation}
u_i \to u_{i+1},\quad  u_i \to v_i, \quad v_i \to u_{i-1}, \quad  u_i \to v_{i+b},\quad  v_{i} \to v_{i+k}\ \mathrm{and}\ v_i \to u_{i-b-k}
\end{equation}
holds. Now, let $\alpha \in \Aut(\G)$ be the unique automorphism fixing $v_{b}$ and mapping $u_0$ to $u_b$. Then Lemma~\ref{le:HAT} implies that $\alpha$ is of order $3$, and so it cyclically permutes the vertices $u_0, u_b, v_{b-k}$, in this order. It must thus also cyclically permute the respective common neighbors $v_{b+k}, u_{b-1}$ and $u_{-k}$ of these vertices with $v_b$. 

Let $C_3 = (u_0,v_b,u_{-k},v_{-k},v_0)$ and $C_4 = (u_0,v_b,u_{-k},v_{1-k},v_1)$ be the two $5$-cycles containing the $2$-path $P = (u_0,v_b,u_{-k})$ (confront Lemma~\ref{le:k=3}). Since the $2$-path $P\alpha = (u_b,v_b,v_{b+k})$ is contained in the $5$-cycle $C = (u_b,v_b,v_{b+k},u_{b+k},v_{2b+k})$ and $v_{b+k} \ot u_{b+k}$ while for the corresponding two edges on $C_3$ and $C_4$ we have $u_{-k} \to v_{-k}$ and $u_{-k} \ot v_{1-k}$, respectively, it follows that $C = C_4\alpha$, and so $v_{1-k}\alpha = u_{b+k}$ and $v_1\alpha = v_{2b+k}$. Since $u_{1-k}$ is the common neighbor of $u_{-k}$ and $v_{1-k}$, we thus also get $u_{1-k}\alpha = u_{b+k-1}$ and similarly $u_1\alpha = v_{2b}$. As $u_{-k} \to v_{-k}$ and we have already determined the $\alpha$ images of $u_{1-k}$ and $v_{b-k}$, it thus follows that $v_{-k}\alpha = v_{b+2k}$. Similarly $v_0\alpha = u_{b+1}$, implying that $C_3\alpha = (u_b,v_b,v_{b+k},v_{b+2k},u_{b+1})$. Moreover, as $v_{-k} \to v_0$ the edge $u_{b+1}v_{b+2k}$ must be a $1$-spoke (note that, since the $5$-cycles are induced, it cannot be a $(b+k)$-spoke), and so $2k = 2$. It follows that $n = 2m$ is even and $k = m+1$, that is $\G = \N(2m;1,b,b+m+1;m-1)$.

Let us now consider the $\Aut(\G)$-orbit $\oo$ of the $2$-path $P$. Since $u_0 \to v_b \to u_{-k}$, the fact that the action of $\Aut(\G)$ on the set of edges of $\G$ is regular implies that for each pair of vertices $x$ and $y$ of $\G$ with $x \to y$ there is a unique neighbor $z$ of $y$ (not adjacent to $x$) such that $(x,y,z) \in \oo$ and a unique neighbor $w$ of $x$ (not adjacent to $y$) such that $(w,x,y) \in \oo$. Considering $P\alpha$ and $P\alpha^2$ we first find that 
$$
(u_i,v_{i+b},u_{i-k}), (u_i,v_i,v_{i+k}), (v_i,v_{i+k},u_{i+k-1}) \in \oo\quad \mathrm{for}\ \mathrm{all}\ i \in \ZZ_n.
$$
Since $(u_{1-k},v_{1-k},v_1)\alpha = (u_{b+k-1},u_{b+k},v_{2b+k})$, edge-transitivity and the above remarks imply that also 
$$
(u_i,u_{i+1},v_{i+b+1}), (v_i,u_{i-1},v_{i-1}), (v_i,u_{i-b-k},u_{i-b-k+1}) \in \oo\quad \mathrm{for}\ \mathrm{all}\ i \in \ZZ_n.
$$
Therefore, since $(v_{b+1},u_b,v_b) \in \oo$ and $(u_b,v_b)\alpha = (v_{b-k},v_b)$, we get $v_{b+1}\alpha = u_{b-k}$, and so the common neighbor $u_{b+1}$ of $u_b$ and $v_{b+1}$ is mapped by $\alpha$ to $u_{b-k-1}$. Similarly $(v_b,u_{b-1},v_{b-1}) \in \oo$, and so $v_{b-1}\alpha = u_{1-k}$, implying that $u_{b-2}\alpha = v_{1-k}$. Thus
$$
u_{-k}\alpha^{-1}\rho^2\alpha = u_{b+1}\alpha = u_{b-k-1} = u_{-k}\rho^{b-1}\ \mathrm{and}\ 
v_{1-k}\alpha^{-1}\rho^2\alpha = u_b\alpha = v_{b-k} = v_{1-k}\rho^{b-1}.
$$ 
Then Lemma~\ref{le:regular} implies $\alpha^{-1}\rho^2\alpha = \rho^{b-1}$. Observe that this implies $\langle 2 \rangle = \langle b-1 \rangle$ in $\ZZ_{2m}$, and so $b$ is odd, say $b = 2b'+1$. 
We can now also completely determine the action of $\alpha$. For instance, for any $i \in \ZZ_n$ we get $u_{2i}\alpha = u_0\rho^{2i}\alpha = u_0\alpha\rho^{i(b-1)} = u_{i(b-1) + b}$. Similarly we find that for each $i \in \ZZ_n$
\begin{equation}
\label{eq:alpha}
	u_{2i}\alpha = u_{i(b-1) + b},\quad u_{2i+1}\alpha = v_{i(b-1) + 2b},\quad v_{2i}\alpha = u_{i(b-1) + b + 1}\quad \mathrm{and}\quad v_{2i+1}\alpha = v_{i(b-1) + 2b + m + 1}.
\end{equation}
Thus $v_b = v_b\alpha = v_{2b'+1}\alpha = v_{b'(b-1) + 2b + m + 1}$, implying that $0 = b'(b-1) + b + m + 1 = 2b'^2 + 2b' + m + 2 = 2b'(b'+1) + m + 2$ (and consequently $b^2+3 = 4b'^2 + 4b' + 4 = 0$). The fact that $n = 2m$ is even now implies that $m$ must be even, so that $4$ divides $n$. In fact, as $b'(b'+1)$ is even $m \equiv 2 \pmod{4}$ must hold, and so $n \equiv 4 \pmod{8}$. In view of the fact that $\langle 2 \rangle = \langle b-1 \rangle$ holds in $\ZZ_n$ we thus also get that $b-1 \equiv 2 \pmod{4}$, and so $b = 4b_0 - 1$ for some $b_0 \geq 1$.

For the converse let $\G = \N(2m;1,b,b+m+1;m-1)$, where $b = 4b_0-1$ for some $b_0 \geq 1$ and $m > 2$ is an even divisor of $b^2+3$ with $m \equiv 2 \pmod{4}$ and $b < 2m$. It is easy to see that the assumptions imply $3 \leq b \leq 2m-5$ and $b+m+1 \notin \{2,2m-1\}$ (when computed modulo $2m$), implying that the edge $u_0u_1$ is on exactly one $3$-cycle. It thus suffices to prove that $\G$ is edge-transitive. Observe that $\gcd(2m,b-1) = \gcd(2m,2(2b_0-1))$, and so the fact that $m$ divides $b^2+3 = 4((2b_0-1)^2+2b_0-1) + 4$ implies that $\langle b-1 \rangle = \langle 2 \rangle$ in $\ZZ_{2m}$. Thus defining $\alpha$ as in (\ref{eq:alpha}) gives rise to a permutation of the vertex set of $\G$. We show that $\alpha$ also preserves adjacency. In view of the action of $\alpha$ this will imply that $\langle \rho, \alpha \rangle \leq \Aut(\G)$ acts transitively on the set of all edges of $\G$, showing that $\G$ is edge-transitive. 

Observe that, since $m \equiv 2 \pmod{4}$ and $0 = b^2+3 = 2(8b_0^2-4b_0+2)$ in $\ZZ_{2m}$, we must have that $8b_0^2 - 4b_0 + 2 = m$ in $\ZZ_{2m}$, and so $2b_0(b-1) = m-2$. It is now easy to verify that $\alpha$ does indeed preserve adjacency. For instance, the vertex $u_{2i}$ is mapped to $u_{i(b-1)+b}$, while the $\alpha$-images of its six neighbors $u_{2i-1}, u_{2i+1}, v_{2i}, v_{2i+1}, v_{2i+b}, v_{2i+b+m+1}$ can be obtained as follows. Since $u_{2i-1} = u_{2(i-1)+1}$, we have $u_{2i-1}\alpha = v_{(i-1)(b-1)+2b} = v_{i(b-1)+b+1}$, which is a neighbor of $u_{i(b-1)+b}$. That $u_{2i+1}\alpha$, $v_{2i}\alpha$ and $v_{2i+1}\alpha$ are neighbors of $u_{2i}\alpha$ follows directly from (\ref{eq:alpha}). Next, as $v_{2i+b} = v_{2(i+2b_0-1)+1}$, we have $v_{2i+b}\alpha = v_{(i+2b_0-1)(b-1)+2b+m+1} = v_{i(b-1)+b}$, which is a neighbor of $u_{i(b-1)+b}$. Finally, let $m = 2m'$ and note that $m'(b-1) = m$ (since $b-1 \equiv 2 \pmod{4}$). Thus the $\alpha$-image of $v_{2i+b+m+1} = v_{2(i+2b_0+m')}$ is $u_{(i+2b_0+m')(b-1)+b+1} = u_{i(b-1)+m-2+m+b+1} = u_{i(b-1)+b-1}$, which is also a neighbor of $u_{i(b-1)+b}$. We leave the remaining adjacencies to the reader.
\end{proof}

Combining together Lemma~\ref{le:lambda>2}, Proposition~\ref{pro:lambda2} and Proposition~\ref{pro:lambda1} proves Theorem~\ref{the:girth3}. It is now also easy to compile a list of all half-arc-transitive Nest graphs of girth $3$ up to a given order. Of course, one needs to check for potential isomorphisms after compiling the list of all parameter sets satisfying Proposition~\ref{pro:lambda1}.  All $46$ pairwise nonisomorphic examples up to $n = 1000$ (that is, up to order $2000$), are given in Table~\ref{tab:HAT} (note that we have ordered the parameters corresponding to the spokes in increasing order, so that the actual parameter $b$ from the above proposition is not always the one following $1$). It is interesting to note that there are nonisomorphic examples of certain orders, namely of orders $728$, $1064$, $1736$ and $1976$.

\begin{table}
$$
{\small
\begin{array}{|c|c|c|}
\hline
\N(28; 1, 6, 19; 13) & \N(52; 1, 7, 34; 25) & \N(76; 1, 15, 54; 37) \\
\N(84; 1, 10, 51; 41) & \N(124; 1, 11, 74; 61) & \N(148; 1, 22, 95; 73) \\
\N(156; 1, 34, 111; 77) & \N(172; 1, 14, 99; 85) & \N(196; 1, 38, 135; 97) \\
\N(228; 1, 15, 130; 113) & \N(244; 1, 27, 150; 121) & \N(268; 1, 59, 194; 133) \\
\N(292; 1, 18, 163; 145) & \N(316; 1, 47, 206; 157) & \N(364; 1, 19, 202; 181) \\
\N(364; 1, 34, 215; 181) & \N(372; 1, 51, 238; 185) & \N(388; 1, 71, 266; 193) \\
\N(412; 1, 94, 299; 205) & \N(436; 1, 91, 310; 217) & \N(444; 1, 22, 243; 221) \\
\N(508; 1, 39, 294; 253) & \N(516; 1, 99, 358; 257) & \N(532; 1, 23, 290; 265) \\
\N(532; 1, 62, 327; 265) & \N(556; 1, 86, 363; 277) & \N(588; 1, 135, 430; 293) \\
\N(604; 1, 66, 367; 301) & \N(628; 1, 26, 339; 313) & \N(652; 1, 118, 443; 325) \\
\N(676; 1, 46, 383; 337) & \N(724; 1, 98, 459; 361) & \N(732; 1, 27, 394; 365) \\
\N(772; 1, 170, 555; 385) & \N(796; 1, 186, 583; 397) & \N(804; 1, 75, 478; 401) \\
\N(844; 1, 30, 451; 421) & \N(868; 1, 51, 486; 433) & \N(868; 1, 135, 570; 433) \\
\N(876; 1, 130, 567; 437) & \N(892; 1, 79, 526; 445) & \N(916; 1, 190, 647; 457) \\
\N(948; 1, 111, 586; 473) & \N(964; 1, 31, 514; 481) & \N(988; 1, 175, 670; 493) \\
\N(988; 1, 138, 631; 493) & &\\
\hline
\end{array}
}
$$
\caption{All half-arc-transitive Nest graphs of girth $3$ up to order $2000$.}
\label{tab:HAT}
\end{table}

The fact that for valence $6$ we get half-arc-transitive generalizations of generalized Petersen graphs for the first time (there are no half-arc-transitive Rose window graphs) is an interesting fact in its own. We now show that there is another reason why the half-arc-transitive Nest graphs of girth $3$ are important. As was mentioned in the introduction, the graphs constructed in~\cite{KutMarSpa10} are half-arc-transitive of valence $12$ with universal reachability relation. However, except for the fact that half-arc-transitive graphs of valence $4$ cannot have universal reachability relation (see~\cite{Mar98}), it was not know whether half-arc-transitive graphs with universal reachability relation of valences smaller than $12$ (that is $6$, $8$ or $10$) exist. As it turns out, the Nest graphs settle this question in the affirmative by providing an infinite family of examples of valence $6$. 

\begin{theorem}
\label{the:universal}
Let $\G = \N(2m;1,b,b+m+1;m-1)$, where $b = 4b_0-1$ for some $b_0 > 1$ and $m > 2$ is an even divisor of $b^2+3$ with $m \equiv 2 \pmod{4}$ and $b < 2m$. Then $\G$ is a half-arc-transitive graph. Moreover, its reachability relation is universal if and only if $3$ does not divide $m$, while in the case that $3$ does divide $m$ the graph $\G$ has three alternets.
\end{theorem}

\begin{proof}
That $\G$ is half-arc-transitive follows from Proposition~\ref{pro:lambda1}. Let $R$ be the reachability relation on $\G$ and let us fix the $\Aut(\G)$-induced orientation of the edges of $\G$ such that $u_0 \to u_1$. From the proof of Proposition~\ref{pro:lambda1} it follows that for each $i$ we have $u_i \to u_{i+1}$, $u_i \to v_i$, $u_i \to v_{i+b}$, $v_i \to u_{i-1}$, $v_i \to u_{i-b-m-1}$ and $v_i \to v_{i+m+1}$. Consider now the alternating path $(u_0,v_0,v_{m-1},u_{m-2},u_{m-3})$. The arc $(u_0,v_0)$ is thus $R$-related to any arc whose tail is of the form $u_i$ or of the form $v_{m-1+i}$, where $i$ is any element of the subgroup $\langle m-3 \rangle$ of $\ZZ_{2m}$. Thus, if $m$ is not divisible by $3$ then $R$ is clearly universal. If however $3$ divides $m$, then $3$ also divides $b$ (recall that $m$ divides $b^2 + 3$). It is now easy to see that an alternating path of even length, starting in $u_0$, can only reach vertices of the form $u_{i}$ and $v_{2+i}$, where $i$ is from the subgroup $\langle 3 \rangle$ of $\ZZ_{2m}$. It is thus clear that $R$ has three equivalence classes, with representatives $(u_0,u_1)$, $(u_1,u_2)$ and $(u_2,u_3)$.
\end{proof}

As mentioned in the introduction, it turns out that the graphs from Theorem~\ref{the:universal} appear also in a recent paper by Zhou and Zhang~\cite{ZhoZha17} (albeit in a somewhat different form) in which the authors classified the half-arc-regular bicirculants of valence $6$ (it should be pointed out, however, that the authors did not investigate the natural orientation induced by the action of their automorphism group and were not aware of the fact that there is an infinite family of examples with universal reachability relation among them). Note that, in view of Proposition~\ref{pro:lambda1}, the automorphism groups of our Nest graphs indeed act regularly on the sets of their edges. Furthermore, in the light of~\cite[Proposition~1.1]{ZhoZha17} our result shows that as long as we restrict to graphs of girth $3$ the automorphism group of a half-arc-transitive bicirculant of valence $6$ necessarily acts regularly on its edge-set, except possibly for examples, in which none of the induced subgraphs on the two orbits of the corresponding semiregular automorphism is connected.


\begin{thebibliography}{}

\bibitem{AntHujKut15} I.~Anton\v ci\v c, A.~Hujdurovi\'c, K.~Kutnar,
		A classification of pentavalent arc-transitive bicirculants,
		{\em J. Algebr. Combin.} {\bf 41} (2015) 643--668.

\bibitem{ArrHubKutOreSpa15} A.~Arroyo, I.~Hubard, K.~Kutnar, E.~O'Reilly, P.~\v Sparl,
		Classification of symmetric Taba\v cjn graphs,
		{\em Graphs Combin.} {\bf 31} (2015) 1137--1153. 
		
\bibitem{CamPraWor93} P.~J.~Cameron, C.~E.~Praeger, N.~C.~Wormald, 
		Highly arc-transitive digraphs and universal covering digraphs, 
		{\em Combinatorica} {\bf 13} (1993) 1--21.
		
\bibitem{DobMalMarNow07} E.~Dobson, A.~Malni\v c, D.~Maru\v si\v c, L.~Nowitz,
		Semiregular automorphisms of vertex-transitive graphs of certain valencies,
		{\em J. Comb. Theory, Ser. B} {\bf 97} (2007) 371--380.
		
\bibitem{FenKwa07} Y.-Q.~Feng, J.~H.~Kwak,
		Cubic symmetric graphs of order a small number times a prime or a prime square,
		{\em J. Comb. Theory, Ser. B} {\bf 97} (2007) 627--646.
		
\bibitem{FruGraWat71} R.~Frucht, J.~E.~Graver, M.~E.~Watkins, 
		The groups of the generalized Petersen graphs,
		{\em Proc. Camb. Philos. Soc.} {\bf 70} (1971) 211--218.
		
\bibitem{GodRoy_book} C.~Godsil, G.~Royle, 
		{\em Algebraic graph theory},
		Springer-Verlag, New York (2001).
		
\bibitem{HujKutMar14} A.~Hujdurovi\'c, K.~Kutnar, D.~Maru\v si\v c,
		Half-arc-transitive group actions with a small number of alternets,
		{\em J. Comb. Theory, Ser. A} {\bf 124} (2014) 114--129.
		
\bibitem{Isa_book} I.~M.~Isaacs,
		{\em Finite group theory},
		Graduate Studies in Mathematics, vol. 92, American Mathematical Society, Providence, RI (2008).
				
\bibitem{Kov04} I.~Kov\'acs,
		Classifying arc-transitive circulants,
		{\em J. Algebr. Combin.} {\bf 20} (2004) 353--358.
		
\bibitem{KovKutMar10} I.~Kov\'acs, K.~Kutnar, D.~Maru\v si\v c,
		Classification of edge-transitive rose window graphs,
		{\em J. Graph Theory} {\bf 65} (2010) 216--231.
		
\bibitem{KovKuzMal10} I.~Kov\'acs, B.~Kuzman, A.~Malni\v c,
		On non-normal arc transitive 4-valent dihedrants,
		{\em Acta Math. Sinica, English Series} {\bf 26} (2010) 1485--1498.
		
\bibitem{KovKuzMalWil12} I.~Kov\'acs, B.~Kuzman, A.~Malni\v c, S.~Wilson,
		Characterization of edge-transitive 4-valent bicirculants,
		{\em J. Graph Theory} {\bf 69} (2012) 441--463.

\bibitem{KutMar09} K.~Kutnar, D.~Maru\v si\v c, 
		A complete classification of cubic symmetric graphs of girth 6,
		{\em J. Comb. Theory, Ser. B} {\bf 99} (2009) 162--184.
		
\bibitem{KutMarSpa10} K.~Kutnar, D.~Maru\v si\v c, P.~\v Sparl,
		An infinite family of half-arc-transitive graphs with universal reachability relation,
		{\em European J. Combin.} {\bf 31} (2010) 1725--1734.
		
\bibitem{Li05} C.~H.~Li,		
		Permutation groups with a cyclic regular subgroup and arc transitive circulants,
		{\em J. Algebr. Combin.} {\bf 21} (2005) 131--136.
		
\bibitem{MalMarSpaFre07} A. Malni\v c, D.~Maru\v si\v c, P.~\v Sparl, B.~Frelih,
		Symmetry structure of bicirculants,
		{\em Discrete Math.} {\bf 307} (2007) 409--414.
		
\bibitem{Mar81} D.~Maru\v si\v c, 
		On vertex symmetric digraphs,
		{\em Discrete Math.} {\bf 36} (1981) 69--81.
		
\bibitem{Mar98} D.~Maru\v si\v c, 
		Half-transitive group actions on finite graphs of valency 4, 
		{\em J. Comb. Theory, Ser. B} {\bf 73} (1998) 41--76.
		
\bibitem{MarPis00} D.~Maru\v si\v c, T.~Pisanski, 
		Symmetries of hexagonal molecular graphs on the torus,
		{\em Croat. Chem. Acta} {\bf 73} (2000) 969--981.
		
\bibitem{Pis07} T.~Pisanski,
		A classification of cubic bicirculants,
		{\em Discrete Math.} {\bf 307} (2007) 567--578.
		
\bibitem{Vas17} G.~Vasiljevi\'c,
		{\em O simetri\v cnih Nest grafih},
		MSc thesis, Faculty of education of the University of Ljubljana, Ljubljana (2017).
		
\bibitem{Ver15} G.~Verret,
		Arc-transitive graphs of valency 8 have a semiregular automorphism,
		{\em Ars Math. Contemp.} {\bf 8} (2015) 29--34.
		
\bibitem{Wil08} S.~Wilson, 
		Rose window graphs,
		{\em Ars Math. Contemp.} {\bf 1} (2008) 7--18.
		
\bibitem{ZhoFen10} J.-X.~Zhou, Y.-Q.~Feng, 
		Tetravalent $s$-transitive graphs of order twice a prime power, 
		{\em J. Aust. Math. Soc.} {\bf 88} (2010) 277--288.
		
\bibitem{ZhoZha17} J.-X.~Zhou, M.-M.~Zhang,
		The classification of half-arc-regular bi-circulants of valency 6,
		{\em European J. Combin.} {\bf 64} (2017) 45--56.
		
\end{thebibliography}
\end{document}